\documentclass[12pt,a4paper]{amsart}

\usepackage{amsmath,amssymb,amsthm,amsfonts}
\usepackage{mathabx}
\usepackage{bm}
\usepackage[shortlabels]{enumitem}
\usepackage[bookmarks=true,hyperindex,pdftex,colorlinks,citecolor=red, linkcolor=blue]{hyperref}
\usepackage{centernot} 
\usepackage{marginnote} 
\usepackage{bbm}
\usepackage[all]{xy}
\usepackage{tikz}
\usetikzlibrary{arrows}
\usetikzlibrary{shapes,decorations}
\usetikzlibrary{positioning}
\usepackage{tikz-cd}
\usepackage{scalerel}
\usepackage{microtype}
\usepackage{stackengine,wasysym}

\newcommand\reallywidetilde[1]{\ThisStyle{%
		\setbox0=\hbox{$\SavedStyle#1$}%
		\stackengine{-.1\LMpt}{$\SavedStyle#1$}{%
			\stretchto{\scaleto{\SavedStyle\mkern.2mu\AC}{.5150\wd0}}{.6\ht0}%
		}{O}{c}{F}{T}{S}%
}}

\DeclareMathOperator{\conv}{conv}                           
\DeclareMathOperator{\Lip}{Lip}                             

\newcommand{\pten}{\ensuremath{\widehat{\otimes}_\pi}}           
\newcommand{\vspan}{\operatorname{span}} 

\newcommand{\N}{\mathbb{N}}             
\newcommand{\R}{\mathbb{R}}             
\newcommand{\C}{\mathbb{C}}             
\newcommand{\K}{\mathbb{K}}
\newcommand{\D}{\mathbb{D}}

\newcommand{\norm}[1]{\left\|{#1}\right\|}                  

\newcommand{\F}{\mathcal{F}}                                

\newcommand{\Hinf}{\mathcal{H}^{\infty}}                              

\renewcommand{\Re}{\operatorname{Re}} 

\AddToHook{cmd/appendix/before}{%
	\setcounter{equation}{0}%
	\setcounter{theorem}{0}%
}

\def\<{\langle}
\def\>{\rangle}

\theoremstyle{plain}
\newtheorem{theorem}{Theorem}[section]
\newtheorem{lemma}[theorem]{Lemma}
\newtheorem{corollary}[theorem]{Corollary}
\newtheorem{proposition}[theorem]{Proposition}

\theoremstyle{definition}
\newtheorem*{definition*}{Definition}
\newtheorem{definition}[theorem]{Definition}

\newtheorem{question}{Question}

\theoremstyle{remark}
\newtheorem{remark}[theorem]{Remark}

\begin{document}

	\title{\textbf{Linearization of holomorphic Lipschitz functions}}

	\author[R. M. Aron]{Richard Aron}
	\address{Department of Mathematical Sciences, Kent State University, Kent, OH 44242
		USA} \email{aron@math.kent.edu}
	
	\author[V. Dimant]{Ver\'onica Dimant}
	\address{Departamento de Matem\'{a}tica y Ciencias, Universidad de San
		Andr\'{e}s, Vito Dumas 284, (B1644BID) Victoria, Buenos Aires,
		Argentina and CONICET} \email{vero@udesa.edu.ar}
	
	\author[L. C. Garc\'ia-Lirola]{Luis C. Garc\'ia-Lirola}
	\address{Departamento de Matem\'{a}ticas, Universidad de Zaragoza, 50009, Zaragoza, Spain}
\email{luiscarlos@unizar.es}

\author[M. Maestre]{Manuel Maestre}
\address{Departamento de An\'{a}lisis Matem\'{a}tico, Universidad de Valencia, Doctor Moliner 50, 46100 Burjasot,
	Valencia, Spain
}
\email{manuel.maestre@uv.es}

\keywords{Banach space; holomorphic function; Lipschitz function; Linearization; Symmetric regularity}

\subjclass[2020]{
	Primary
	46E15,  
	46E50;  
	Secondary
	46B20,  
	46B28	
}

\maketitle

\begin{abstract}
	
	Let $X$ and $Y$ be complex Banach spaces with $B_X$ denoting the open unit ball of $X.$ This paper studies various aspects of the {\em holomorphic Lipschitz space} $\mathcal HL_0(B_X,Y)$, endowed with the Lipschitz norm. This space
	is the intersection of the spaces, $\Lip_0(B_X,Y)$ of Lipschitz mappings  and  $\mathcal H^\infty(B_X,Y)$ of
	bounded holomorphic  mappings,  from $B_X$ to $Y$.
	Thanks to the Dixmier-Ng theorem, $\mathcal HL_0(B_X, \mathbb C)$ 
	is indeed a dual space, whose predual $\mathcal G_0(B_X)$ shares linearization properties with both the Lipschitz-free space and Dineen-Mujica predual of $\mathcal H^\infty(B_X)$. We explore the similarities and differences between these spaces, and combine techniques to study the properties of the space of holomorphic Lipschitz functions. In particular, we get that $\mathcal G_0(B_X)$ contains a 1-complemented subspace isometric to $X$  and that $\mathcal G_0(X)$ has the (metric) approximation property whenever $X$ has it. We also  analyze when $\mathcal G_0(B_X)$ is a subspace of $\mathcal G_0(B_Y)$, and we obtain an analogous to Godefroy's characterization of functionals with a unique norm preserving extension to the holomorphic Lipschitz context.
\end{abstract}

\section{Introduction}

Linearizing non-linear functions is a typical procedure in infinite dimensional analysis. Originating nearly 70 years ago with Grothendieck \cite{Gro} (and his research about linearization of bilinear mappings through the projective tensor product), the practice of identifying spaces of continuous \emph{non-linear}  functions with spaces of continuous \emph{linear} mappings defined on Banach spaces has  proved to be a useful technique.
Accordingly, the study of geometric and topological properties of these \emph{linearizing Banach spaces} has increasingly attracted interest.

Lipschitz functions (defined on pointed metric spaces) and holomorphic bounded functions (defined on the open unit ball of a Banach space) are really different both as sets and as function spaces. However, when looking at their linearization processes several similarities emerge. The purpose of this article is to study, in light of these resemblances, the new set of functions consisting of the intersection of the previous sets. Lipschitz holomorphic functions defined on  the open unit ball of a Banach space taking the value 0 at 0 will be our focus of attention. In the exploration of this set we take advantage of a result of Ng \cite{Ng} concerning the existence of preduals and all the background about related linearization processes.

We begin with a brief review of important terms and concepts. General references for Lipschitz functions include \cite{GK} and \cite{Weaver} and a standard reference for holomorphic functions on finite or infinite dimensional domains is \cite{MuLibro}. The linearization process for bounded holomorphic functions is originally developed in \cite{Mu}. A review about linearization procedures both for Lipschitz functions and for bounded holomorphic functions appeared in the recent survey \cite{GHT} while a general approach to linearizing non-linear sets of functions was settled in \cite{CZ}.

For a metric space $(M,d)$ and a Banach space $Y,$ let $\Lip(M,Y)$ be the vector space of all  $f\colon M \to Y$ such that $\norm{f(x_1) - f(x_2)} \leq C d(x_1,x_2)$ for some $C > 0$ and for all $x_1 \neq x_2 \in M$.
The smallest $C$ in the above definition is the {\em Lipschitz constant} of $f, L(f)$. Let $0 \in M$ denote an arbitrary fixed point. In order to get a normed space, we will be particularly interested in the subspace $\Lip_0(M,Y)$ consisting of those $f \in \Lip(M,Y)$ such that $f(0) = 0.$ In this way, $L(f) = 0$ if and only if $f = 0$, and so $\norm{\cdot}=L(\cdot)$ defines a norm on $\Lip_0(M,Y)$. \\

For complex Banach spaces $X$ and $Y$ and open set $U\subset X$,  denote by $\mathcal H^\infty(U,Y)$ the vector space of all $f\colon U \to Y$ such that $f$  is  holomorphic (i.e.\ complex
Fr\'echet differentiable)   and bounded  on  $U$, endowed with the supremum  norm. 
In both the Lipschitz and $\mathcal H^\infty$ situations, if the range $Y = \K,$ then the notation is shortened to $\Lip_0(M)$ and $\mathcal H^\infty(U)$.

It is known that  $\Lip_0(M)$ and $\mathcal H^\infty(U)$ are dual spaces and that in some special situations, the predual is unique. The construction of a (or, in some cases, the) predual follows the same lines for both the Lipschitz and $\mathcal H^\infty$ situations: Calling $X$ one of $\Lip_0$ or $\mathcal H^\infty,$ we consider those functionals $\varphi \in X^*$ such that $\varphi|_{\overline{B}_X}$ is continuous when $\overline{B}_X$ is endowed with the compact-open topology. Among such functionals are the evaluations $f \leadsto \delta(x)(f) \equiv f(x)$ where $x$ ranges over the domain of $f \in X.$ In the case of $\Lip_0(M),$ the closed span of the set of such $\varphi$ will be denoted $\mathcal F(M)$  while the analogous closed span for $\mathcal H^\infty(U)$ is $\mathcal G^\infty(U).$ Each of these is a Banach space, being a closed subspace of $\Lip_0(M)^*,$ and $\mathcal H^\infty(U)^*,$ respectively. Using a standard technique developed by Ng \cite {Ng}, it follows that $\mathcal F(M)^* \equiv \Lip_0(M)$ and $\mathcal G^\infty(U)^* \equiv \mathcal H^\infty(U).$\\

Among the most important common features of $\Lip_0$ and $\mathcal H^\infty$ is {\em linearization}. In each of the two cases below, $\delta$ is the evaluation inclusion taking $x \leadsto \delta(x).$ Also, for
$f$ in either $\Lip_0(M,Y)$ or $\mathcal H^\infty(U,Y),$ $T_f$ is the unique linear mapping making the diagram commute. Moreover, $\|f\| = \|T_f\|.$

\begin{equation*}
	\xymatrix{
		M \ar[r]^f  \ar[d]_{\delta}    &  Y  \\
		\mathcal{F}(M)  \ar[ru]_{T_f}  &
	} \qquad\qquad \xymatrix{
		U \ar[r]^f  \ar[d]_{\delta}    &  Y  \\
		\mathcal{G}^\infty(U)  \ar[ru]_{T_f}  &
	}
\end{equation*}

\subsubsection*{Notation}

$X, Y$ will stand for complex Banach spaces. We denote by $B_X$ (respectively $S_X$) its open unit ball (respectively unit sphere). $\mathcal L(X, Y)$ denotes the space of continuous linear maps from $X$ to $Y$, and $X^*=\mathcal L(X, \mathbb C)$. $\mathcal P(^m X, Y)$ stands for the space of continuous $m$-homogeneous polynomials, that is, those $P\colon X\to Y$ so that  there exists a continuous  $m$-linear symmetric map $\widecheck{P}\colon X\times \cdots \times X\to Y$ with $P(x)=\widecheck{P}(x,\ldots, x)$. We also write $\mathcal P(^m X)=\mathcal P(^m X, \mathbb C)$. We say that $P\in \mathcal P(^m X, Y)$ is of \emph{finite type} if   $P(x)=\sum_{j=1}^n [x_j^*(x)]^m y_j$ for certain $x_j^*\in X^*$ and $y_j\in  Y$. $\mathcal P_f(^m X, Y)$ stands for the space of finite type $m$-homogeneous polynomials. Moreover, we set $\mathcal P(X, Y)$ (resp. $\mathcal P_f(X, Y)$) to be the space of finite sums of continuous  homogeneous polynomials (resp. homogeneous polynomials of finite type) from $X$ to $Y$. Also, $\mathbb D(z,r)$ (resp. $C(z,r)$) denotes the open disc (resp. the circumference) in $\C$ centered at $z$ with radius $r$, in particular $\D=\D(0,1)$.

Recall that $X$ is said to have the Bounded Approximation Property (BAP) if there is $\lambda>0$ such that the identity $I\colon X\to X$ can be approximated by finite-rank operators in $\lambda B_{\mathcal L(X,X)}$ uniformly on compact sets (equivalently, pointwise). If $\lambda=1$, then $X$ is said to have the Metric Approximation Property (MAP). If $X$ has $\lambda$-BAP and $Y$ is $\lambda'$-complemented in $X$, then $Y$ has $\lambda\lambda'$-BAP. Recall, also, the version of this notion without control of the norms: $X$ has the  Approximation Property (AP) if  the identity $I\colon X\to X$ can be approximated by finite-rank operators in $\mathcal L(X,X)$ uniformly on compact sets. We refer the reader to \cite{Casazza} for examples and applications.

\subsubsection*{Organization of the paper}

Section \ref{sec:HL} introduces the main space of interest, $\mathcal HL_0(B_X,Y),$ consisting of those  functions that are in both $\Lip_0(B_X,Y)$ and $\mathcal H^\infty(B_X,Y).$ A number of properties of $\mathcal HL_0(B_X,Y)$ are discussed and it is proved that this space really differs from $\Lip_0(B_X,Y)$ and $\mathcal H^\infty(B_X,Y)$ (in the sense that a nonseparable space can be injected in between). Then we focus on the predual $\mathcal G_0(B_X)$  of $\mathcal HL_0(B_X)$ (where $Y = \C$). Specifically, we will see that $\mathcal HL_0(B_X)$ has a canonical predual whose properties echo those of $\mathcal H^\infty(B_X)$ and $\Lip_0(B_X).$ When $X =\C$ with open unit disc $\D,$ one consequence of our work is a characterization of the extreme points of the closed ball of $\mathcal HL_0(\D)$ and of the norm attaining elements of $\mathcal HL_0(\D)$ considered as the dual of $\mathcal G_0(\D)$. In Section \ref{sec:AP} we deal with  the (metric) approximation property for $\mathcal G_0(B_X)$, again inspired by the results for $\mathcal G^\infty(B_X)$.  The final two sections involve a closer inspection of $\mathcal G_0(B_X)$ and its relationship with $\mathcal G_0(B_{X^{**}})$.  Section \ref{section: relation} begins by considering the interaction between $\mathcal G_0(B_X)$ and $\mathcal G_0(B_Y)$ when $X \subset Y$ and then focuses on the case of $X \subset X^{\ast\ast}$. The final section studies a natural connection between $\mathcal G_0(B_{X^{**}})$ and $\mathcal G_0(B_X)^{**}$ under the hypothesis of $X^{**}$ having the MAP. Among other things, this enables us to characterize, under natural conditions on $X$ and $X^{\ast\ast},$  when a function $f \in \mathcal HL_0(B_X)$ has a unique norm preserving extension to $\mathcal HL_0(B_{X^{\ast\ast}}).$	
Both sections make use of the concept of \emph{(Arens) symmetric regularity,} which is reviewed in Section \ref{section: relation}.

\section{The space of holomorphic Lipschitz functions and its predual}\label{sec:HL}

In the case that the metric space $M$ is  $B_X$, the open unit ball of a complex Banach space $X$, and  $Y$ is another complex Banach space,
$\Lip_0(B_X, Y)$ is the space of Lipschitz functions $f\colon B_X\to Y$ with $f(0)=0$
and:
$$
L(f)=\sup\left\{ \frac{\|f(x)-f(y)\|}{\|x-y\|}:\ x\neq y \in B_X\right\}.
$$

It is well known that $L(\cdot)$ defines a norm on $\Lip_0(B_X,Y)$ and $(\Lip_0(B_X,Y),L(\cdot))$ is a Banach space. Indeed, $\Lip_0(B_X,Y)$ is isometrically isomorphic to the space of operators $\mathcal L(\mathcal F(B_X), Y)$, where $\mathcal F(B_X)$ denotes the \emph{Lipschitz-free space} over $B_X$ (see e.g. \cite{GodSurvey, Weaver}).

Next, $\Hinf(B_X, Y)$ stands for the space of bounded holomorphic functions from $B_X$ to $Y$, which is a Banach space when endowed with the supremum norm.  Analogous to the Lipschitz case above, we have that $\Hinf(B_X,Y)$ is isometrically isomorphic to $\mathcal L(\mathcal G^\infty(B_X), Y)$, where $\mathcal G^\infty(B_X)$ is Mujica's canonical predual of $\Hinf(B_X)$ \cite{Mu} (we will review the space $\mathcal G^\infty(B_X)$ later in this section).

The parallel behavior of these Lipschitz and $\mathcal H^\infty$ spaces was the authors' motivation to introduce and study the following space and its canonical predual:
\begin{align*}
	\mathcal HL_0(B_X,Y) & =\Lip_0(B_X, Y)\cap \Hinf(B_X,Y).
\end{align*}
We will also denote $\mathcal{H}L_0(B_X)=\mathcal HL_0(B_X, \mathbb C)$. Sometimes we will deal with holomorphic Lipschitz functions without assuming $f(0)= 0$, and then  we use the notation $\mathcal {H}L(B_X, Y)$ and $\mathcal {H}L(B_X)$.

Since both normed spaces $\Hinf(B_X,Y)$ and $\Lip_0(B_X,Y)$ are complete (with their respective norms) and  each $f\in \mathcal HL_0(B_X,Y)$ satisfies $\|f\|_\infty\le L(f)$ we easily derive that $\mathcal HL_0(B_X,Y)$ is a Banach space with norm $L(\cdot)$.

Given $f\in \Hinf(B_X,Y)$ such that $df\in\Hinf(B_X,\mathcal L(X,Y))$ and $f(0)=0$,  by the Mean Value Theorem, we have that $\|f(x)-f(y)\|\le \|df\| \|x-y\|$ for any $x,y \in B_X.$
Then, $f\in\Lip_0(B_X,Y)$ and $L(f)\le \|df\|$. Conversely, if $f\in \mathcal{H}L_0(B_X,Y)$ we know that $df\in\mathcal H(B_X, \mathcal L(X,Y))$. Also, for $x,y\in B_X$,
$$
\|df(x)(y)\| = \lim_{h\to 0}\left\|\frac{f(x+hy)-f(x)}{h}\right\|\le L(f).
$$ This means that $df$ belongs to $\Hinf(B_X, \mathcal L(X,Y))$ and $\|df\|\le L(f)$.

This shows that there is another useful representation of our primary space of interest.
\begin{proposition}  \label{prop: HL0}$\mathcal{H}L_0(B_X,Y)=\{f\in\Hinf(B_X,Y):\ df\in\Hinf(B_X,\mathcal L(X,Y));\ f(0)=0\}$.\ \ Moreover, for every $f \in \mathcal HL_0(B_X,Y), \ L(f) = \norm{df};$ that is, $L(f) = \sup_{x \in B_X} \norm{df(x)}.$
\end{proposition}	

\begin{remark} \label{rem: differential}
	If $f\colon B_X\to Y$ is a holomorphic function and $x \in B_X$ then $f(x+h)=\sum_{m=1}^\infty P_m(x)(h)$ for  $h$ in a suitable neighborhood of 0, where $P_m$ is an $m$-homogeneous polynomial. Recall that $df(x)(h)=P_1(x)(h)$ for every $h\in X$.    
\end{remark}

Note that $P|_{B_X}\in \mathcal {H}L_0(B_X, Y)$ for every $P\in \mathcal P(X,Y)$ such that $P(0)=0$, a fact that will be useful later.

When $Y = \C,$ we can define a mapping
\begin{align*}
	\Phi\colon \mathcal HL_0(B_X) & \to \Hinf (B_X,X^*)\\
	f & \mapsto df
\end{align*}

In general, $\Phi$ is an isometry {\em into} $\mathcal H^\infty(B_X,X^*),$ although if $X$ also equals  $ \C$, then $\Phi$ is {\em onto}. Indeed, in the one-dimensional case, $\Phi$
is surjective since every holomorphic function $f$ on $\mathbb D$ has a primitive that is Lipschitz whenever $f$ is bounded.   However, $\Phi$ is not surjective for $X\neq \mathbb C$. Indeed, given $P\in\mathcal P(^2X)$, we have that $P|_{B_X}\in \mathcal{H}L_0(B_X)$ and $dP\in\mathcal L(X, X^*)$ is symmetric (i.e. $dP(x)(y)=dP(y)(x)$ for every $x,y\in X$). Note that $df$ is linear only when $f$ is a 2-homogeneous polynomial.
Hence, a non-symmetric element of $\mathcal L(X, X^*)$ (which always exists whenever the dimension of $X$ is strictly bigger than one) cannot be in the range of $\Phi$.

In particular, we see that
\[\mathcal HL_0(\D) = \{ f \in \mathcal H^\infty(\D): \  f( 0)= 0 \ {\rm and \ } f^\prime \in \mathcal H^\infty(\D) \}.\]
A lot of research has been done on $\mathcal HL_0(\D)$ and on $\mathcal
HL_0(U)$ for certain domains $U \subset \C^n$ such as the Euclidean ball. See, e.g.,
\cite{AhSc,BLS,BBLS2018,BBLS2019,BT,Du,Pa,Ru}  where this topic is approached from different viewpoints than what is done here.

For the case of  $\mathcal HL(B_X,Y)$ we consider the norm $\|f\|_{\mathcal HL}=\max\{\|f(0)\|, L(f)\}$. The fact that this is a norm and that $(\mathcal HL(B_X,Y), \|\cdot\|_{\mathcal HL})$ is a Banach space follows easily. 
Note that $\|f\|_\infty\le 2\|f\|_{\mathcal HL}$ for any $f\in \mathcal HL(B_X,Y)$. Also, it is plain to see that $\mathcal HL_0(B_X,Y)$ is a 1-complemented subspace of $\mathcal HL(B_X,Y)$. Moreover, motivated by a similar result for $\Lip_0$-spaces (see \cite[Th. 1.7.2]{Weaver}) we get:

\begin{proposition}\label{prop: G vs G_0} Let $X, Y$ be complex Banach spaces. Then $\mathcal HL(B_X, Y)$ is isometric to a $1$-complemented subspace of $\mathcal HL_0(B_{X\oplus_1 \mathbb C}, Y)$.
\end{proposition}

\begin{proof}
	Consider $\Phi\colon \mathcal HL(B_X, Y)\to \mathcal HL_0(B_{X\oplus_1\mathbb C}, Y)$ given by $\Phi f(x,\lambda)=f(x)+(\lambda-1)f(0)$. It is easy to check that $\Phi f$ is Lipschitz with $L(\Phi f)\leq \norm{f}_{\mathcal HL}$ for every $f\in \mathcal HL(B_X, Y)$. Note that
	\[ L(\Phi f)\geq \sup\left\{\frac{\norm{\Phi f(x,0)-\Phi f(y,0)}}{\norm{x-y}}: x\neq y\in B_X\right\}=L(f)\]
	and also
	\[ L(\Phi f)\geq \frac{\norm{\Phi f(0,1)-\Phi f(0,0)}}{\norm{(0,1)-(0,0)}_1}=\norm{f(0)}, \]
	so we actually have $L(\Phi f)= \norm{f}_{\mathcal HL}$. Thus $\Phi$ is an into isometry.
	
	Now consider $T\colon \mathcal HL_0(B_{X\oplus_1\mathbb C}, Y)\to \mathcal HL(B_X, Y)$ given by $Tg(x)=g(x,0)+g(0,1)$. One can easily check that $\norm{T}\leq 1$ and $T\circ \Phi=I_{\mathcal HL(B_X,Y)}$. Therefore $P=\Phi \circ T$ is a norm-one projection from $\mathcal HL_0(B_{X\oplus_1\mathbb C}, Y)$ onto $\Phi(\mathcal HL(B_X, Y))$.
	
\end{proof}

Note that there are plenty of examples of non-Lipschitz functions in $\mathcal H^\infty (\mathbb D)$. For instance, given a sequence $(b_n)\subset \mathbb C\setminus\{1\}$ with $|b_n|=1$ and $b_n\to 1$, define $f\colon \{b_n\}\cup\{1\}\to\mathbb C$ by $f(1)=0$ and $f(b_n)=\sqrt{|b_n-1|}$. Then the Rudin-Carleson theorem provides an extension of $f$  which lies in the disc algebra $\mathcal A(\mathbb D)$ (that is, the space of uniformly continuous functions in $\mathcal H^\infty(\mathbb D)$) and has the same supremum norm, but it is not Lipschitz.

Our next goal is to show that $\mathcal HL_0(B_X)$ is indeed much smaller than both $\mathcal H^\infty(B_X)$ and $\Lip_0(B_X)$. More precisely, we will prove the following result, where we denote $\Hinf_0(B_X) =\{f\in\Hinf(B_X):\  f(0)=0\}$.

\begin{theorem} \label{th:embedl_infty} Let $X$ be a non-null complex Banach space.  Then
	\begin{enumerate}[(a)]
		\item $\ell_\infty$ is isomorphic to a subspace of $\mathcal H^\infty_0(B_X)\setminus \mathcal HL_0(B_X)\cup \{0\}$.
		\item $\ell_\infty$ is isomorphic to a subspace of $\Lip_0(B_X)\setminus \mathcal HL_0(B_X)\cup \{0\}$.
	\end{enumerate}
\end{theorem}

We will provide a different proof of Theorem \ref{th:embedl_infty}. $(a)$ in the Appendix. Indeed, there we build   an isomorphism into its image $F\colon \ell_\infty \longrightarrow \Hinf(B_X)$ such that, additionally,  its restriction to $c_0$ satisfies that $F|_{c_0}\colon c_0\longrightarrow \mathcal{A}(B_X)$.

\begin{proof}
	$(a)$ For the case $X=\mathbb C$, it has been observed to the authors that  a classical result is that  given $(z_j)$ an interpolating sequence on $\D$ there exists an topological isomorphism $S\colon\ell_\infty\longrightarrow \Hinf(\D)$  such that $S(c)(z_j)=c_j$ for every $j$ and every $c=(c_n)\in \ell_\infty$ (see e.g. \cite[Theorem VII.2.1 and applications, p.~285]{Garnett} where it is made for the upper half plane), we can also get that $S(c)(0)=0$. Hence, if  $\N$ is  partitioned  into infinitely many infinite sequences $n_{ik}$ with $i, k\in\mathbb N$ and for  $c\in\ell_\infty$ it is  defined $x_c\in\ell_\infty$, $x_c(n_{ik})=(-1)^i c_k$, and
	$Y=\{S(x_c): c\in\ell_\infty\}$, then $Y$ is an subspace of $\Hinf(\D)$ isomorphic to $\ell_\infty$.
	And if $c\neq 0$, then $c_k\neq 0$ for some $k$, and $S(x_c)$ takes values $\pm c_k$ along a sequence tending to 1. So it cannot be  uniformly continuous, hence it is not Lipschitz. 
	
	Now, to get the general case, 
	we fix $x_0\in S_X$ and consider $x^\ast\in X^\ast$ such that $x^\ast(x_0)=1=\|x^\ast\|$. We define
	\[
	\Psi\colon \Hinf(\D)\longrightarrow \Hinf(B_X)
	\]
	by $\Psi(f)=f\circ x^\ast$. Clearly $\Psi$ is a well-defined linear mapping and since $x^\ast(B_X)=\D$ we have that $\Psi$  is an isometry onto its image.
	Moreover, considering its restriction we are going to have
	\[
	\Psi\colon \mathcal HL(\D)\longrightarrow \mathcal HL(B_X),
	\]
	that again is an isometry, now with the Lipschitz norms. Indeed,
	if $f \in \mathcal HL(\D)$ then \[L(\Psi(f))=L(f\circ x^*)\leq L(f)L(x^*)=L(f).\] 
	But
	if $\lambda, \mu\in \D$, then
	\begin{align*}
		|f(\lambda)-f(\mu)|= & |f\circ x^\ast(\lambda x_0)-f\circ x^\ast(\mu x_0)|=|\Psi(f)(\lambda x_0)-\Psi(f)(\mu x_0)|  \\
		\leq &  L(\Psi(f))\|\lambda x_0-\mu x_0\|=L(\Psi(f))|\lambda-\mu|,
	\end{align*}
	and we get $L(f)\leq L(\Psi(f))$.
	Finally, due to the injectivity of $\Psi$ (a direct proof is also elementary) we have that
	\[
	\Psi (\Hinf_0(\D)\setminus \mathcal HL(\D))\subset \Hinf_0(B_X)\setminus \mathcal HL(B_X).
	\]
	Now the claim follows.

	$(b)$   First we consider the 1-dimensional case $X=\C$. Let $l\colon \R\to [0,1]$ be a $C^\infty$ function such that  $l$ is strictly increasing on $(1/2,1)$, $0<l(x)<1$  if $1/2<x<1$, $l(x)=0$ for  $x\le 1/2$, $l(x)=1$ for  $x\ge 1$, $l^{(k)}(1/2)=0$ if $k\ge 0$ and $l^{(k)}(1)=0$ if $k\ge 1$.
	Define $f\colon \C\to [0,1]$ as $f(z)=l(|z|)$. Considered it as being defined  on $\R^2$, $f$ is $C^\infty$ and $df\colon \R^2\to \R$ is a continuous function. Hence, by the Mean Value Theorem, $f\in \Lip_0(\D)$.
	Now we define $T\colon \mathcal HL_0(\D)\to \Lip_0(\D)$ as $T(g)=f\cdot g$. We claim that $T$ is an isomorphism onto its image. Indeed,
	given $g\in\mathcal  HL_0(\D)$ and $z,u\in \mathbb D$,
	\[
	|f(z)g(z)-f(u)g(u)|\leq |f(z)-f(u)||g(z)|+|f(u)||g(z)-g(u)|\leq  2 L(f)L(g)|z-u|.
	\]
	Thus $T$ is a continuous linear mapping with $\|T\|\leq 2L(f)$. Now we check that $T$ is bounded below. As $f(x,y)=l(\sqrt{x^2+y^2})$ we have that $df(x,y)=0$ if $z=x+iy$ satisfies  $|z|\geq 1$. By continuity on a compact set, given $\varepsilon>0$ there exists $0<r<1$ such that if $|z|\geq r$, then both $\|df(x,y)\|<\varepsilon$ and $f(z)>1-\varepsilon$. Thus,  for $g\in \mathcal HL_0(\D)$,
	\[ L(fg)=\|d(fg)\|_\D=\|gdf+fg'\|_\D\geq \|fg'\|_{\D\setminus r\D}-\|g\|_\D\|df\|_{\D\setminus r\D}\geq \|fg'\|_{\D\setminus r\D}-L(g)\varepsilon.
	\]
	But, by the maximum modulus theorem
	\[
	\|fg'\|_{\D\setminus r\D}\geq (1-\varepsilon)\|g'\|_{\D\setminus r\D}=(1-\varepsilon)\|g'\|_{\D}=(1-\varepsilon)L(g).
	\]
	and we get $L(fg)\geq (1-2\varepsilon)L(g)$, for every $\varepsilon>0$. As a consequence
	\[ L(Tg)=L(fg)\geq L(g),
	\]
	and $T$ is bounded below. Moreover, $T(g)=f\cdot g$ is never holomorphic on $\D$ for any $g\in \mathcal  HL_0(\D)\setminus \{0\}$, and
	$T(\mathcal HL_0)(\D)$ is isomorphic to $\mathcal HL_0(\D)$ which in turn is isometric to $\mathcal H^\infty(\D)$ that has a subspace isomorphic to $\ell_\infty$.
	
	The general case is a straightforward consequence of the above argument in the following natural way.  Let $X$ a non-null complex Banach space and take $x^\ast \in S_{X^\ast}$.
	Defining $R\colon\Lip_0(\D)\to \Lip_0(B_X)$ by $R(h)=h\circ x^\ast$, we are going to have that $R$ is an isometry into. Hence, $R\circ T\colon\mathcal  HL_0(\D)\to \Lip_0(B_X)$  is an isomorphism into its image and we get that $\ell_\infty$ is isomorphic to a subspace of
	$\mathcal HL_0(B_X)$.  But if $g \in\mathcal  HL_0(\D)\setminus \{0\}$, then $R\circ T(g)=(f\cdot g)x^\ast$ is not a Gateaux holomorphic function since its restriction to $\{z x\, :z\in \D\}$ is not holomorphic. We conclude that $\ell_\infty\setminus \{0\}\subset \Lip_0(B_X)\setminus \mathcal  HL_0(B_X)$.
\end{proof}

In this rest of the section, we will focus the attention on the canonical predual of space $\mathcal{H}L_0(B_X)$ and show that it shares many  properties with the canonical preduals of $\mathcal H^\infty(B_X)$ and $\Lip_0(B_X)$.

Let us denote by $\tau_0$ the compact-open topology on $\mathcal HL_0(B_X)$. An easy argument using Montel's theorem \cite[Th. 15.50]{Dirichlet} and Remark \ref{rem: differential} shows that $\overline{B}_{\mathcal HL_0(B_X)}$ is $\tau_0$-compact. In fact, on this ball, convergence in the topology $\tau_0$ coincides with pointwise convergence. Thus, the Dixmier-Ng theorem \cite{Ng} says that $\mathcal HL_0(B_X)$ is a dual space with predual given by
$$
\mathcal{G}_0(B_X) :=\{\varphi\in \mathcal HL_0(B_X)^*:\ \varphi|_{\overline B_{\mathcal HL_0(B_X)}} \textrm{ is } \tau_0-\textrm{continuous}\}.
$$
For $x\in B_X$ and $f\in \mathcal HL_0(B_X)$, denote $\delta(x)(f)=f(x)$. Clearly $\delta(x)\colon \mathcal HL_0(B_X)\to \mathbb C$ is linear and continuous meaning that $\delta(x)\in \mathcal HL_0(B_X)^*$. Also, $\delta(x)|_{\overline B_{\mathcal HL_0(B)}}$  is $\tau_0$-continuous so $\delta(x)\in \mathcal G_0(B_X)$.

\begin{proposition}\label{prop:diagram} Let $X$ be a complex Banach space.
	\begin{enumerate}[(a)]
		\item The mapping
		\begin{align*}
			\delta\colon B_X & \to \mathcal{G}_0(B_X) \\
			x & \mapsto \delta(x)
		\end{align*} is holomorphic and $\norm{\delta(x)-\delta(y)}=\norm{x-y}$ for every $x,y\in B_X$. In particular, $\delta\in \mathcal HL_0(B_X,\mathcal{G}_0(B_X)) $ with $L(\delta)=1$.
		\item $\mathcal G_0(B_X)= \overline{\vspan} \{\delta(x):\ x\in B_X\}$.
		\item For any complex Banach space $Y$ and any $f\in \mathcal HL_0(B_X,Y)$, there is a unique operator $T_f\in \mathcal L(\mathcal G_0(B_X), Y)$ such that the following diagram commutes:
		\begin{equation*}
			\xymatrix{
				B_X \ar[r]^f  \ar[d]_{\delta}    &  Y  \\
				\mathcal{G}_0(B_X)  \ar[ru]_{T_f}  &
			}
		\end{equation*}
		The map $f\mapsto T_f$ defines an isometric isomorphism from $\mathcal HL_0(B_X, Y)$ onto $\mathcal L(\mathcal G_0(B_X), Y)$. These properties characterize $\mathcal G_0(B_X)$ uniquely up to an isometric isomorphism.
		\item A bounded net $(f_\alpha)\subset \mathcal HL_0(B_X)$ is weak-star convergent to a function $f\in\mathcal HL_0(B_X)$ if and only if $f_\alpha(x)\to f(x)$ for every $x\in B_X$.
	\end{enumerate}
\end{proposition}

\begin{proof}
	$(a)$ The map $\delta$ is weakly holomorphic since for any $f\in \mathcal G_0(B_X)^*=\mathcal HL_0(B_X)$ we have that $f\circ \delta=f$ is holomorphic. Thus, $\delta$ is holomorphic (see \cite[ Th. 8.12]{MuLibro}). Also, given $x,y\in B_X$, we have
	\[ \norm{\delta(x)-\delta(y)}=\sup_{f\in B_{\mathcal HL_0(B_X)}}|\<f,\delta(x)-\delta(y)\>|=\sup_{f\in B_{\mathcal HL_0(B_X)}} |f(x)-f(y)|\leq \norm{x-y},\]
	and equality holds since we may take $f=x^*|_{B_X}$ where $\norm{x^*}=1$ and $x^*(x-y)=\norm{x-y}$.
	
	$(b)$ Just observe that for every $f\in \mathcal HL_0(B_X)=\mathcal G_0(B_X)^*$ we have that $f=0$ whenever $f|_{\{\delta(x): x\in B_X\}}=0$.

	$(c)$ First, note that an interpolation argument shows that the set $\{\delta(x):x\in B_X\setminus\{0\}\}$ is linearly independent in $\mathcal G_0(B_X)$. Indeed, assume that $\sum_{j=1}^n \lambda_j \delta(x_j)=0$ for different points $x_j\in B_X\setminus\{0\}$ and $\lambda_j\in \mathbb C$. Let $x_0=0$ and $\lambda_0=0$. Take $x_{ij}^*\in S_{X^*}$ with $x_{ij}^*(x_i-x_j)=\norm{x_i-x_j}$ and define $f(x)=\sum_{j=0}^n \overline{\lambda_j}\prod_{i\neq j} \frac{x_{ij}^*(x_i-x)}{\norm{x_i-x_j}}$. Then $f\in \mathcal HL_0(B_X)$ and  $0=\<f, \sum_{j=1}^n \lambda_j \delta(x_j)\>=\sum_{j=1}^n |\lambda_j|^2$.
	
	Now, given $f\in \mathcal HL_0(B_X,Y)$, we define $T_f(\delta(x)):=f(x)$ for every $x\in B_X$ (this is the only possibility to get a commutative diagram) and extend it linearly to $\vspan\{\delta(x):x\in B_X\}$. Note that, given $u=\sum_{j=1}^n \lambda_j \delta(x_j)$,
	\begin{align*} \norm{T_f u} &= \norm{\sum_{j=1}^n \lambda_j f(x_j)}=\sup_{y^*\in B_{Y^*}}\left|\sum_{j=1}^n \lambda_j (y^*\circ f)(x_j)\right| =\sup_{y^*\in B_{Y^*}} \left|\<u, y^*\circ f\>\right|\\
		&\leq \sup\{L(y^*\circ f): y^*\in B_Y\}\norm{u} = L(f)\norm{u}.
	\end{align*}
	Thus, $T_f$ extends uniquely to an operator $T_f\in \mathcal L(\mathcal G_0(B), Y)$ with $\norm{T_f}\leq L(f)$. Since $L(\delta)=1$ and $f=T_f\circ \delta$, indeed we get that $\norm{T_f}=L(f)$.
	
	Moreover, the map $f\mapsto T_f$ is onto since, given any $T\in L(\mathcal G_0(B_X), Y)$, we have that $f:= T\circ \delta$ is a holomorphic Lipschitz map with $f(0)=0$ and $T=T_f$.
	
	The uniqueness of $\mathcal G_0(B_X)$ follows from the diagram property and the fact that $\norm{T_f}=L(f)$.
	
	$(d)$ The ball $\overline{B}_{\mathcal HL_0(B_X)}$ is $\tau_0$-compact and the  weak-star topology is coarser than $\tau_0$, so they coincide on $\overline{B}_{\mathcal HL_0(B_X)}$.
\end{proof}

\begin{proposition} \label{prop: X 1-complemented}
	For every complex Banach space $X$ we have that $X$ is isometric to a 1-complemented subspace of $\mathcal{G}_0(B_X)$.
\end{proposition}
\begin{proof}
	In the particular case of $f=Id\colon B_X\to X$, differentiating the diagram in Proposition~\ref{prop:diagram} and using that $d(Id) (x)=Id$ for all $x\in B_X$, we obtain another commutative diagram where all the arrows are linear:
	\begin{equation*}
		\xymatrix{
			X \ar[r]^{Id}  \ar[d]_{d\delta (0)}    &  X  \\
			\mathcal{G}_0(B_X)  \ar[ru]_{T_{Id}}  &
		}
	\end{equation*}
	Moreover, $d\delta(0)$ is an isometry. Indeed, given $x\in X$ and $f\in \mathcal HL_0(B_X)$ we have
	\[\<f, d\delta(0)(x)\> = \lim_{t\to 0} \< f, \frac{\delta(tx)-\delta(0)}{t}\> =\lim_{t\to 0} \frac{f(tx)-f(0)}{t} = df(0)(x)\]
	and so
	\[\norm{d\delta(0)(x)}=\sup\{|df(0)(x)|: f\in B_{\mathcal HL_0(B_X)}\}\leq \norm{x}.\]
	The other inequality is clear due to the commutative diagram:
	$$
	\|x\|=\|T_{Id}\circ d\delta(0) (x)\|\le \|d\delta(0) (x)\|.
	$$

	Finally, let $P=d\delta(0)\circ T_{Id}$. Then, using that $T_{Id}\circ d\delta(0)=Id$, we have
	\[ P^2= d\delta(0)\circ T_{Id}\circ d\delta(0)\circ T_{Id}= d\delta(0)\circ T_{Id}= P,\]
	so $P$ is a norm-one projection from  $\mathcal{G}_0(B_X)$ onto $d\delta(0)(X)$.
\end{proof}

Note that this result also holds for $\mathcal{G}^\infty(B_X)$ \cite{Mu} but not in general for $\F(B_X)$. In \cite{GK} it is proved that this is true for $X$ separable  although for nonseparable $X$ it could even occur that $\F(B_X)$ does not contain a subspace isomorphic to $X$. Another useful property of Lipschitz-free spaces is the fact that they contain a complemented copy of $\ell_1$ \cite{CDW}, the same holds for $\mathcal G_0(B_X)$.

\begin{proposition} Let $X$ be a complex Banach space. Then there is a complemented subspace of $\mathcal G_0(B_X)$ isomorphic to $\ell_1$.
\end{proposition}

\begin{proof}
	$\ell_\infty$ is isomorphic to a subspace of $\mathcal H^\infty(\mathbb D)$. Since $\mathcal H^\infty(\mathbb D)$ is isometric to $\mathcal HL_0(\mathbb D)$, which is a complemented subspace in $\mathcal HL_0(B_X)$, the same holds for $\mathcal HL_0(B_X)$. It is a classical result (see \cite[Th. 4]{BP}) that this implies its predual $\mathcal G_0(B_X)$ contains a complemented copy of $\ell_1$.
\end{proof}

Next, we want to describe the closed unit ball of $\mathcal{G}_0(B_X)$. For that, we introduce some more notation.
We denote by $\conv$ the convex hull of a set and by $\Gamma$  the absolute convex hull of a set.
As usual in the Lipschitz world, for every $x,y\in B_X$ with $x\neq y$,   $m_{x,y}$ stands for the \emph{elementary molecule} $\frac{\delta(x)-\delta(y)}{\|x-y\|}$.   Also, for every $x\in B_X$, $y\in X$ and $f\in \mathcal HL_0(B_X)$, we denote  $e_{x,y}(f):=df(x)(y)$. Then $e_{x,y}\in \mathcal G_0(B_X)$ with   $\norm{e_{x,y}}=\norm{y}$. Indeed, it is clear that
\[ \norm{e_{x,y}}=\sup\{|df(x)(y)|: f\in B_{\mathcal HL_0(B_X)}\}\leq \sup\{\norm{df(x)}: f\in B_{\mathcal HL_0(B_X)}\}\norm{y}\leq \norm{y}.\]
Conversely, take $x^*\in X^*$ with $x^*(y)=\norm{y}$ and $\norm{x^*}=1$. Then $x^*|_{B_X}\in \mathcal HL_0(B_X)$ and $e_{x,y}(x^*|_{B_X})=x^*(y)=\norm{y}$. This shows that $e_{x,y}$ belongs to $\mathcal HL_0(B_X)^*$ and the equality of norms. Finally, by a simple application of a Cauchy's integral formula we derive that the restriction of $e_{x,y}$ to $\overline B_{\mathcal HL_0(B_X)}$ is $\tau_0$-continuous and so it belongs to $ \mathcal G_0(B_X)$.

\begin{proposition} \label{prop:capsula convexa}
	Let $X$ be a complex Banach space. Then,
	$$\overline{B}_{\mathcal{G}_0(B_X)} = \overline\Gamma \{m_{x,y}:\ x,y\in B_X,\ x\neq y \} = \overline\conv \{e_{x,y}:\ x\in B_X, y\in S_X \}$$
\end{proposition}
\begin{proof}
	By Proposition~\ref{prop:diagram}, we have that $\norm{m_{x,y}}=1$ for every $x,y\in B_X$ with $x\neq y$. Also,
	\[L(f)=\sup\{|\<f, m_{x,y}\>|: x,y\in B_X, x\neq y\} \, \text{for all } f\in \mathcal HL_0(B_X).\]
	Thus, $\{m_{x,y}:x,y\in B_X, x\neq y\}$ is $1$-norming for $\mathcal HL_0(B_X)$. Equivalently, $\overline{B}_{\mathcal{G}_0(B_X)} = \overline\Gamma \{m_{x,y}:\ x,y\in B_X,\ x\neq y \}$. Analogously, we have that
	\[ L(f)=\norm{df}=\sup\{\norm{df(x)}:x\in B_X\}=\sup\{|\<f, e_{x,y}\>| : x\in B_X, y\in S_X\}\]
	and so $\overline{B}_{\mathcal{G}_0(B_X)} = \overline{\Gamma\{e_{x,y}:x\in B_X, y\in S_X\}}$. But $e_{x,\lambda y_1+\eta y_2}= \lambda e_{x,y_1}+\eta e_{x,y_2}$ for every $\lambda, \eta\in \mathbb C$ so actually $\overline{B}_{\mathcal{G}_0(B_X)} = \overline{\conv}\{e_{x,y}:x\in B_X, y\in S_X\}$.
\end{proof}

As a consequence, the density characters of $X$ and $\mathcal G_0(B_X)$ coincide. In particular $X$ is separable if and only if $\mathcal{G}_0(B_X)$ is separable.

We will now relate $\mathcal G_0(B_X)$ with the Lipschitz-free space $\mathcal F(B_X)$ and Mujica's predual $\mathcal G^\infty(B_X)$ of $\mathcal H^\infty(B_X)$. Note that each element of $\mathcal F(B_X)$ can be seen also as an element of $\mathcal G_0(B_X)$, but maybe with a different behavior. For instance, consider  $z\in B_X\setminus \{0\}$ and $\mu$ given by $\<\mu,f\>=\int_{C(0,1)} f(\lambda z)d\lambda$ for $f\in \Lip_0(B_X)$.  Then $\mu\neq 0$ in $\mathcal F(B_X)$ but  $\<\mu,f\>=0$ for all $f\in \mathcal HL_0(B_X)$, so $\mu=0$ when considered as an element of $\mathcal G_0(B_X)$. The next proposition formalizes  this situation. We say that an operator $T\colon X\to Y$ is a \emph{quotient operator} if $T$ is surjective and $\norm{y}= \inf\{\norm{x}:Tx=y\}$ for every $y\in Y$; this implies that $X/\ker T$ is isometrically isomorphic to $Y$.

\begin{proposition} \label{prop: quotient}
	Let $X$ be a complex Banach space.
	\begin{enumerate}[(a)]
		\item The operator
		\begin{align*}\pi\colon \mathcal F(B_X)&\to \mathcal{G}_0(B_X)\\
			\delta(x)&\mapsto \delta(x)
		\end{align*}
		is a quotient operator with kernel $\mathcal HL_0(B_X)_\bot=\{\mu\in \mathcal F(B_X): \<f, \mu\>=0 \,\forall f\in \mathcal HL_0(B_X)\}$. Thus $\mathcal G_0(B_X)\equiv \mathcal F(B_X)/\mathcal HL_0(B_X)_\bot$ isometrically.
		
		\item The operator
		\begin{align*}\Psi\colon \mathcal{G}^\infty(B_X)\pten X&\to \mathcal{G}_0(B_X)\\
			\delta(x)\otimes y&\mapsto e_{x,y}
		\end{align*}
		is a quotient map with $\norm{\Psi}=1$. In addition, the operator $\Psi$ is injective if  and only if $X= \mathbb C$.
	\end{enumerate}
\end{proposition}
\begin{proof}
	$(a)$ First note that the existence of such an operator $\pi$ follows from the linearization property of Lipschitz-free spaces applied to the 1-Lipschitz map $B_X\to \mathcal G_0(B_X)$ given by $x\mapsto \delta(x)$. Also, $\pi^*\colon \mathcal HL_0(B_X) \to \Lip_0(B_X)$ is just the inclusion map since
	\[ \pi^*f(x)=\<\pi^*f, \delta(x)\>=\<f, \pi(\delta(x))\>=\<f,\delta(x)\>=f(x) \quad \forall f\in \mathcal HL_0(B_X),\forall x\in B_X.\]
	Thus, $\pi^*$ is an isometry into. It is a standard fact that this implies that $\pi$ is a quotient operator. Moreover, $\ker\pi = \pi^*(\mathcal HL_0(B_X))_\bot= \mathcal HL_0(B_X)_\bot$.

	$(b)$ Consider the into isometry
	\begin{align*}
		\Phi\colon \mathcal HL_0(B_X) & \to \Hinf (B_X,X^*)\\
		f & \mapsto df
	\end{align*}
	defined after Proposition \ref{prop: HL0}. Recall that $\mathcal{G}^\infty(B_X)\pten X$ is a predual of $\mathcal L(\mathcal{G}^\infty(B_X), X^*)\simeq \Hinf (B_X,X^*)$ (see e.g. \cite{Ryan}). Thus, if we restrict $\Phi^*$ to this predual we obtain $\Psi=\Phi^*|_{\mathcal{G}^\infty(B_X)\pten X}$, note that $\Psi(\delta(x)\otimes y)=e_{x,y}\in \mathcal G_0(B_X)$ for all $x,y$ and so $\Psi(\mathcal G^\infty(B_X)\pten X)\subset \mathcal G_0(B_X)$. Then $\norm{\Psi}=1$ and $\Psi$ is a quotient operator since $\Psi^*=\Phi$ is an into isometry. In the case $X=\mathbb C$, we have indeed that $\Phi\colon \mathcal HL_0(\mathbb D)\to\mathcal H^\infty (\mathbb D)$ is an onto isometry, and thus $\Psi$ is also an isometry from $\mathcal G^\infty (\mathbb D)$ onto $\mathcal G_0(\mathbb D)$. However, $\Psi$ is not injective for $X\neq \mathbb C$ since $\Phi$ is not surjective. 
\end{proof}

Thus, $\mathcal G_0(\mathbb D)$ is isometric to $\mathcal G^\infty (\mathbb D)$ (which is the unique predual of $\mathcal H^\infty(\mathbb D)$ \cite{Ando}). We have some immediate consequences.

\begin{corollary} A function $f$ is an extreme point of  $\overline{B}_{\mathcal HL_0(\mathbb D)}$ if and only if $f'$ is an extreme point of $\overline{B}_{\mathcal H^\infty(\mathbb D)}$. \end{corollary}

\begin{corollary} A function $f\in \mathcal HL_0(\mathbb D)$ attains its norm as a functional on $\mathcal G_0(\mathbb D)$ if and only if $f'\in \mathcal H^\infty (\mathbb D)$ attains its norm as a functional on $\mathcal G^\infty (\mathbb D)$.
\end{corollary}

Let us state one more consequence of Proposition \ref{prop: quotient}.

\begin{corollary}
	Let $X$ be a complex Banach space and $\varphi\in \mathcal G_0(B_X)$.
	\begin{enumerate}[(a)]
		\item There are sequences $(x_n),(y_n)\subset B_X$ with $x_n\neq y_n$ and $(a_n)\subset \ell_1$ such that
		\[ \varphi=\sum_{n=1}^\infty a_n m_{x_n,y_n}.\]
		Moreover, $\norm{\varphi}=\inf\sum_{n=1}^\infty|a_n|$ where the infimum is taken over all such representations of $\varphi$.
		\item There are sequences $(x_n)\subset B_X$, $(y_n)\subset S_X$ and $(a_n)\subset \ell_1$ such that
		\[ \varphi=\sum_{n=1}^\infty a_n e_{x_n,y_n}.\]
		Moreover, $\norm{\varphi}=\inf\sum_{n=1}^\infty|a_n|$ where the infimum is taken over all such representations of $\varphi$.
	\end{enumerate}
\end{corollary}

\begin{proof}
	Given $\varepsilon>0$, Proposition \ref{prop: quotient} $(a)$ provides an element $\mu\in \mathcal F(B_X)$ with $\pi(\mu)=\varphi$ and $\norm{\pi(\mu)}\leq\norm{\varphi}+\varepsilon$. It is known (see e.g.  \cite[Lem. 3.3]{APPP}) that there are points $x_n, y_n\in B_X$ and $(a_n)\subset \ell_1$ with $\mu=\sum_{n=1}^\infty a_n\frac{\delta(x_n)-\delta(y_n)}{\norm{x_n-y_n}}$ and $ \sum_{n=1}^\infty |a_n|\leq \norm{\mu}+\varepsilon\leq \norm{\varphi}+2\varepsilon$ (here $\delta$ denotes the canonical embedding $\delta\colon B_X\to\mathcal F(B_X)$). Then $\varphi=\sum_{n=1}^\infty a_n \pi(\frac{\delta(x_n)-\delta(y_n)}{\norm{x_n-y_n}})=\sum_{n=1}^\infty a_n m_{x_n, y_n}$.
	
	Item $(b)$ follows similarly using the corresponding property for projective tensor products (see e.g. \cite[Prop. 2.8]{Ryan}) and $\mathcal G^\infty(B_X)$ \cite[Th. 5.1]{Mu2}.
\end{proof}

Another consequence of the linearization process  shows that functions in $\mathcal HL_0$ behave similarly to functions in $\Lip_0(B_X,B_Y)$ (that can be isometrically factorized through the free-Lipschitz spaces $\mathcal F(B_X)$ and $\mathcal F(B_Y)$). Given $f\in \mathcal HL_0(B_X,Y)$ with $f(B_X)\subset B_Y$ we can easily obtain a commutative diagram:

\begin{equation} \label{diagrama cuadrado}
	\xymatrix{
		B_X \ar[r]^{f}  \ar[d]_{\delta_X }    &  B_Y \ar[d]^{\delta_Y }  \\
		\mathcal{G}_0(B_X)  \ar[r]_{ T_{\delta_Y \circ f}}   &  \mathcal{G}_0(B_Y),
	}
\end{equation} where $T_{\delta_Y \circ f}$ is linear and $\|T_{\delta_Y \circ f}\|=L(f)$.

\section{Approximation properties on $\mathcal G_0(B_X)$} \label{sec:AP}

Following Mujica's ideas \cite{Mu} we devote this section to study the metric approximation property (MAP) and the approximation property (AP) for $\mathcal{G}_0(B_X)$ whenever $X$ has the same property. Beginning with the MAP, we  prove the following result about approximation of elements in the closed unit ball of the dual space. We first introduce the notation:
\begin{itemize}
	\item $\mathcal P_0(X,Y)$: The vector space of polynomials  $P\colon X\to Y$ such that $P(0)=0$  endowed with the norm $\|dP\|=L(P|_{B_X})$.
	\item $\mathcal P_{f,0}(X,Y)$: The subspace of $\mathcal P_0(X,Y)$ consisting of finite type polynomials.
\end{itemize}

\begin{proposition} \label{Prop:bola unidad}
	Let $X$ and $Y$ be complex Banach spaces. Then
	
	\begin{enumerate}[(a)]
		\item $\overline{B}_{\mathcal HL_0(B_X,Y)}= \overline{B_{\mathcal P_0(X,Y)}}^{\tau_0}$.
		\item If $X$ has the MAP then $\overline{B}_{\mathcal HL_0(B_X,Y)}= \overline{B_{\mathcal P_{f,0}(X,Y)}}^{\tau_0}$.
	\end{enumerate}
\end{proposition}

\begin{proof}
	$(a)$ If $f\in \overline{B}_{\mathcal HL_0(B_X,Y)}$  then $f\in \mathcal H^\infty(B_X,Y)$ and $f(0)=0$.  Consider the Taylor series expansion of $f$ at 0: $f(x)=\sum_{k=0}^\infty P^kf(0)(x)$. As in \cite{Mu}, for each $m\in \N\cup\{0\}$, we denote
	$$
	S_mf(x)=\sum_{k=0}^m P^kf(0)(x) \qquad \textrm{and}\qquad \sigma_mf(x)=\frac{1}{m+1} \sum_{k=0}^m S_kf(x).
	$$
	Since $df=\sum_{k=0}^\infty dP^kf(0)\in\mathcal H^\infty(B_X, \mathcal L(X,Y))$ it follows from \cite[Prop.~5.2]{Mu} that $\sigma_mf(x)\to f(x)$ for all  $x\in B_X$ and \[\|d\sigma_mf\|=\|\sigma_m(df)\|\le \|df\|\le 1.\]
	This implies that $f\in \overline{B_{\mathcal P_0(X,Y)}}^{\tau_0}$.
	
	For the reverse inclusion, let $f\in \mathcal HL_0(B_X, Y)$ and  $(P_\alpha)\subset B_{\mathcal P_0(X,Y)}$ such that $P_\alpha (x)\to f(x)$ for all $x\in B_X$. Then $L(f)\le 1$ and so $f\in \overline{B}_{\mathcal HL_0(B_X,Y)}$.
	
	$(b)$ If $X$ has the MAP there is a net of finite rank operators $(T_\alpha)\subset \mathcal L(X,X)$ such that $T_\alpha (x)\to x$ for all $x\in X$ and $\|T_\alpha\|\le 1$ for every $\alpha$. Given $P\in B_{\mathcal P_0(X,Y)}$ we have that $P\circ T_\alpha$ belongs to $B_{\mathcal P_{f,0}(X,Y)}$ (since $L(P\circ T_\alpha|_{B_X})< 1$) and $P(T_\alpha x)\to P(x)$ for every $x$.  This means that $P\in\overline{B_{\mathcal P_{f,0}(X,Y)}}^{\tau_0}$. Finally, an appeal to $(a)$ yields the result.
\end{proof}

\begin{theorem}\label{theor: MAP}
	$X$ has the MAP if and only if $\mathcal{G}_0(B_X)$ has the MAP.
\end{theorem}

\begin{proof} $X$ being isometric to a 1-complemented subspace of $\mathcal{G}_0(B_X)$ it is clear that $X$ has the MAP
	when $\mathcal{G}_0(B_X)$ has it.
	
	Now, suppose that $X$ has the MAP and consider the mapping $\delta\in \overline{B}_{\mathcal HL_0(B_X,\mathcal{G}_0(B_X))}$. By Proposition \ref{Prop:bola unidad} there exist a net $(P_\alpha)\subset B_{\mathcal P_{f,0}(X,\mathcal{G}_0(B_X))}$ such that $P_\alpha(x)\to \delta(x)$ for all $x\in B_X$. Applying a linearization as in Proposition \ref{prop:diagram} we obtain finite rank  linear mappings $T_{P_\alpha}$ with norm bounded by 1, such that the following diagram commutes:
	\begin{equation*}
		\xymatrix{
			B_X \ar[r]^{P_\alpha}  \ar[d]_{\delta}    &  \mathcal{G}_0(B_X)  \\
			\mathcal{G}_0(B_X)  \ar[ru]_{T_{P_\alpha}}  &
		}
	\end{equation*}
	
	Note that $T_{P_\alpha}(\delta(x))=P_\alpha(x)\to \delta(x)=Id(\delta(x))$. Then, we have that $T_{P_\alpha}\to Id$ on $\vspan\{\delta(x):x\in B_X\}$. Since the net $(T_{P_\alpha})$ is bounded the same holds for the closure. Hence, $\mathcal{G}_0(B_X)$ has the MAP.
\end{proof}

Note that our arguments cannot be adapted to the case in which $X$ has the BAP since the approximations of the identity could send the unit ball $B_X$ to a bigger ball (and, hence, we cannot control the Lipschitz norm of $P\circ T_\alpha|_{B_X}$ as in Proposition \ref{Prop:bola unidad} $(b)$).

\begin{question} Does $\mathcal{G}_0(B_X)$ have the BAP whenever $X$ has the BAP?
\end{question}

The same question for $\mathcal{G}^\infty(B_X)$  was posed by Mujica in \cite{Mu}. As far we know, this question is still open.

In contrast to this unknown case about the BAP, the analogous statement for the AP (Approximation Property -without bounds-) was successfully solved by Mujica \cite{Mu} for $\mathcal{G}^\infty(B_X)$. We now turn to this goal for our space $\mathcal{G}_0(B_X)$, following Mujica's scheme but somewhat simplifying the arguments.

Note that in the results about the MAP we used several times that a bounded net of linear operators converges uniformly on compact sets if and only if it converges pointwise on a dense set. For the AP we cannot make use of this kind of argument so our first step will be to describe a locally convex topology $\tau_\gamma$ such that the following topological isomorphism holds:
\begin{equation} \label{tau gamma}
	(\mathcal HL_0(B_X,Y),\tau_\gamma)\cong (\mathcal L(\mathcal{G}_0(B_X),Y),\tau_0).
\end{equation}

\begin{remark}
	Note that for a topology $\tau_\gamma$ satisfying \eqref{tau gamma}, if $(f_\alpha)$ is a bounded net in $ \mathcal HL_0(B_X,Y)$ which converges pointwise to $f\in  \mathcal HL_0(B_X,Y)$ then $f_\alpha\overset{\tau_\gamma}{\to} f$. Indeed, linearizating we obtain a bounded net $(T_{f_\alpha}) \subset \mathcal L(\mathcal{G}_0(B_X),Y)$ which converges pointwise to $T_f$. Then, $T_{f_\alpha}\overset{\tau_0}{\to} T_f$ implying that $f_\alpha\overset{\tau_\gamma}{\to} f$.
	
	As a consequence, we derive from Proposition \ref{Prop:bola unidad} $(a)$ the following identity: \begin{equation}\label{bola tau gamma}
		\overline{B}_{\mathcal HL_0(B_X,Y)}= \overline{B_{\mathcal P_0(X,Y)}}^{\tau_\gamma}.  
	\end{equation}
\end{remark}

In order to work with the $\tau_0$-topology in $\mathcal L(\mathcal{G}_0(B_X),Y)$ it would be good to have a useful description of the compact sets of the space $\mathcal{G}_0(B_X)$. For that, we appeal to the following variation of the classical Grothendieck description of compact sets (which can be proved, for instance, by slightly modifying the proof of \cite[Prop. 9, pg 134]{Ro}):

\begin{lemma}
	Let $X$ be a Banach space and $V\subset S_X$ such that $\overline B_X=\overline\Gamma (V)$. For each compact set $K\subset X$ there exist sequences $(\alpha_j)\in c_0$ (with $\alpha_j>0$ for all $j$) and $(v_j)\subset V$ such that $K\subset \overline\Gamma (\{\alpha_jv_j\})$.
\end{lemma}

A direct consequence of this lemma, along with Proposition \ref{prop:capsula convexa} is the following:

\begin{corollary} \label{cor: compacto en G0}
	Let $K\subset \mathcal{G}_0(B_X)$ be a compact set. Then there exist sequences $(\alpha_j)\in c_0$ and $(x_j,y_j)\subset B_X\times B_X$ (with $\alpha_j>0$ and $x_j\not= y_j$ for all $j$) such that $K\subset \overline\Gamma (\{\alpha_j m_{x_jy_j}\})$. 
\end{corollary}

Now we can introduce, as in \cite[Th. 4.8]{Mu}, a topology $\tau_\gamma$ satisfying \eqref{tau gamma}.

\begin{theorem}\label{teo: top tau delta}
	Let $\tau_\gamma$ be the locally convex topology on  $\mathcal HL_0(B_X,Y)$ generated by the seminorms
	$$
	p(f)=\sup_j \alpha_j\frac{\|f(x_j)-f(y_j)\|}{\|x_j-y_j\|}
	$$ where $(\alpha_j)\in c_0$, $(x_j,y_j)\subset B_X\times B_X$ and $\alpha_j>0$, $x_j\not= y_j$ for all $j$. Then, the mapping
	\begin{eqnarray*}
		(\mathcal HL_0(B_X,Y),\tau_\gamma)&\to &(\mathcal L(\mathcal{G}_0(B_X),Y),\tau_0)\\
		f &\mapsto & T_f
	\end{eqnarray*} is a topological isomorphism.
\end{theorem}

\begin{proof}
	If $K\subset \mathcal{G}_0(B_X)$ is a compact set, by the previous corollary there are sequences $(\alpha_j)\in c_0$, $(x_j,y_j)\subset B_X\times B_X$ with $\alpha_j>0$, $x_j\not= y_j$ for all $j$, such that $K\subset \overline\Gamma (\{\alpha_j m_{x_jy_j}\})$. Then, for all $f\in \mathcal HL_0(B_X,Y)$,
	$$
	\sup_{u\in K} \|T_f u\|\le \sup_j \|T_f(\alpha_j m_{x_jy_j})\|= \sup_j \alpha_j\frac{\|f(x_j)-f(y_j)\|}{\|x_j-y_j\|},
	$$ showing that the mapping $f\mapsto T_f$ is $\tau_\gamma-\tau_0$ continuous.
	
	To prove the continuity of the inverse mapping note that for a seminorm $p$ of $\tau_\gamma$, the associated sequence $\alpha_j m_{x_jy_j}$ converges to 0 in $\mathcal{G}_0(B_X)$. Thus, the set $K=\{\alpha_j m_{x_jy_j}\}\cup\{0\}$ is a compact set in $\mathcal{G}_0(B_X)$ and
	$p(f)=\sup_j \|T_f(\alpha_j m_{x_jy_j})\|=\sup_{u\in K} \|T_f u\|$.  
\end{proof}

Let us state separately the corresponding result for Lipschitz-free spaces, that we will not use but it might be of independent interest. 

\begin{theorem} Let $M$ be a complete pointed metric space. Then
	\begin{enumerate}[(i)]  \item For each compact subset $K$ of $\mathcal F(M)$ there exists sequences $(\alpha_j)\in c_0$ and $(x_j,y_j)\subset M\times M$ (with $\alpha_j>0$ and $x_j\not= y_j$ for all $j$) such that $K\subset \overline\Gamma (\{\alpha_j m_{x_jy_j}\})$. 
		\item Given a Banach space $Y$, let $\tau_\gamma$ be the locally convex topology on  $\Lip_0(M,Y)$ generated by the seminorms
		$$
		p(f)=\sup_j \alpha_j\frac{\|f(x_j)-f(y_j)\|}{d(x_j,y_j)}
		$$ where $(\alpha_j)\in c_0$, $(x_j,y_j)\subset M\times M$ and $\alpha_j>0$, $x_j\not= y_j$ for all $j$. Then, the mapping
		\begin{eqnarray*}
			(\mathcal \Lip_0(M,Y),\tau_\gamma)&\to &(\mathcal L(\mathcal{F}(M),Y),\tau_0)\\
			f &\mapsto & T_f
		\end{eqnarray*} is a topological isomorphism.
	\end{enumerate}
\end{theorem}

Now we examine the relationship between the topologies $\tau_\gamma$ and $\tau_0$ in  $\mathcal HL_0(B_X,Y)$.

\begin{proposition}\label{prop: equivalent topologies}
	Let $X$ and $Y$ be complex Banach spaces. Then, $\tau_\gamma$ is finer than $\tau_0$ in $\mathcal HL_0(B_X,Y)$, and these topologies are equivalent in $\mathcal P(^mX,Y)$ for each $m\in\N$.
\end{proposition}

\begin{proof}
	If $K\subset B_X$ is a compact set, then $\delta(K)\subset \mathcal{G}_0(B_X)$ is compact. By Corollary \ref{cor: compacto en G0}, there exist sequences $(\alpha_j)\in c_0$ and $(x_j,y_j)\subset B_X\times B_X$ (with $\alpha_j>0$ and $x_j\not= y_j$ for all $j$) such that $\delta(K)\subset \overline\Gamma (\{\alpha_j m_{x_jy_j}\})$. Hence, for all $f\in \mathcal HL_0(B_X,Y)$,
	$$
	\sup_{x\in K} \|f(x)\|\le \sup_j \alpha_j\frac{\|f(x_j)-f(y_j)\|}{\|x_j-y_j\|},
	$$ proving the first assertion.
	
	For the second statement, take a seminorm $p$ that generates $\tau_\gamma$: $p(f)=\sup_j \alpha_j\frac{\|f(x_j)-f(y_j)\|}{\|x_j-y_j\|}$,  with $(\alpha_j)\in c_0$, $(x_j,y_j)\subset B_X\times B_X$, $\alpha_j>0$ and $x_j\not= y_j$ for all $j$. For a homogeneous polynomial $P\in \mathcal P(^mX,Y)$ we have:
	
	\begin{eqnarray*}
		p(P) &=& \sup_j \alpha_j\frac{\|P(x_j)-P(y_j)\|}{\|x_j-y_j\|} = \sup_j \frac{\|P(\alpha_j^{1/m}x_j)-P(\alpha_j^{1/m}y_j)\|}{\|x_j-y_j\|}\\
		&=& \sup_j \frac{\left\|\sum_{k=1}^m \binom{m}{k}\widecheck{P}\left((\alpha_j^{1/m}(x_j-y_j))^k,  (\alpha_j^{1/m}y_j)^{m-k}\right)\right\|}{\|x_j-y_j\|}\\
		&=& \sup_j \left\|\sum_{k=1}^m \binom{m}{k}\widecheck{P}\left(\left(\frac{\alpha_j^{1/m}(x_j-y_j)}{\|x_j-y_j\|^{1/k}}\right)^k,  (\alpha_j^{1/m}y_j)^{m-k}\right)\right\|.
	\end{eqnarray*}
	Note that there exist compact sets $K_1$ and $K_2$ in $X$ such that $\left\{\alpha_j^{1/m}\frac{(x_j-y_j)}{\|x_j-y_j\|^{1/k}}\right\}\subset K_1$ and $\left\{\alpha_j^{1/m}y_j\right\}\subset K_2$ (since both sequences go to 0). Then,
	$$
	p(P)\le \sum_{k=1}^m \binom{m}{k} \sup_{a\in K_1, b\in K_2}\|\widecheck{P}(a^k,b^{m-k})\|.
	$$ Using the polarization formula, for each $k\in\{1,\dots, m\}$,
	$$
	\widecheck{P}(a^k,b^{m-k})=\frac{1}{2^m m!}\sum_{\varepsilon_i=\pm 1}\varepsilon_1\cdots \varepsilon_m P\left(\left(\sum_{i=1}^k \varepsilon_i\right)a+\left(\sum_{i=k+1}^m \varepsilon_i\right)b\right).
	$$
	Taking into account that the following set is compact
	$$
	C(K_1,K_2)=\left\{\left(\sum_{i=1}^k \varepsilon_i\right)a+\left(\sum_{i=k+1}^m \varepsilon_i\right)b:\ a\in K_1, b\in K_2, k\in\{1,\dots, m\}, \varepsilon_i=\pm 1\right\},
	$$ and that
	$$
	\sup_{a\in K_1, b\in K_2}\|\widecheck{P}(a^k,b^{m-k})\|\le \frac{1}{m!}\sup_{u\in C(K_1,K_2)}\|P(u)\|
	$$
	we derive the intended inequality:
	$$
	p(P)\le \frac{2^m-1}{m!}\sup_{u\in C(K_1,K_2)}\|P(u)\|.
	$$
\end{proof}

We can now combine all the pieces of our study of the topology $\tau_\gamma$ to obtain the following:

\begin{proposition} \label{prop: densidad AP}
	If $X$ has the AP, for a given $f\in \mathcal HL_0(B_X,Y)$ there exists a net $(P_\alpha)\subset \mathcal P_{f,0}(X,Y)$ such that $P_\alpha\overset{\tau_\gamma}{\to} f$.
\end{proposition}

\begin{proof}
	It is enough to consider $f\in \overline B_{\mathcal HL_0(B_X,Y)}$. Moreover, taking into account the equality \eqref{bola tau gamma} we  just need to prove the result for each homogeneous polynomial $P\in\mathcal P(^mX,Y)$ (for any $m$). Applying \cite[Lem.~5.3]{Mu} (or  composing the polynomial with the approximations of the identity supplied by the AP of $X$) we obtain a net $(P_\alpha)\subset \mathcal P_{f,0}(X,Y)$ such that $P_\alpha\overset{\tau_0}{\to} P$. Now,  Proposition \ref{prop: equivalent topologies} implies that $P_\alpha\overset{\tau_\gamma}{\to} P$, which finishes the proof.
\end{proof}

Finally, we are in the position of proving the announced result:

\begin{theorem}\label{teo: equivalencia AP}
	$X$ has the AP if and only if $\mathcal{G}_0(B_X)$ has the AP.
\end{theorem}

\begin{proof}
	One implication is clear because $X$ is isometric to a complemented subspace of $\mathcal{G}_0(B_X)$.
	
	For the other, take $\delta\in  \mathcal HL_0(B_X,\mathcal{G}_0(B_X))$. By Proposition \ref{prop: densidad AP} there exists a net $(P_\alpha)\subset \mathcal P_{f,0}(X,\mathcal{G}_0(B_X))$ such that $P_\alpha\overset{\tau_\gamma}{\to} \delta$. By the linearization process, appealing to the isomorphism \eqref{tau gamma}, we obtain that $(T_{P_\alpha})\subset \mathcal L(\mathcal{G}_0(B_X),\mathcal{G}_0(B_X))$ is a net of finite rank linear mappings satisfying $T_{P_\alpha}\overset{\tau_0}{\to} Id$.
\end{proof}

\begin{remark}
	
	With the same procedure as at the beginning of the previous section we can produce a canonical predual $\mathcal{G}(B_X)$ of $\mathcal HL(B_X)$ made up of elements of $\mathcal HL(B_X)^*$ which are $\tau_0$-continuous when restricted to the closed unit ball. The fact that $\mathcal HL_0(B_X)$ is a 1-complemented subspace of $\mathcal HL(B_X)$ and that the projection from $\mathcal HL(B_X)$ onto $\mathcal HL_0(B_X)$ is $\tau_0-\tau_0$ continuous allow us to derive that $\mathcal{G}_0(B_X)$ is isometric to a 1-complemented subspace of $\mathcal{G}(B_X)$.
	
	With standard adaptations most of the results of this and the previous sections can be stated for $\mathcal{G}(B_X)$ instead of $\mathcal{G}_0(B_X)$. That is the case of Propositions~\ref{prop:diagram}, \ref{prop: X 1-complemented}, \ref{Prop:bola unidad} and Theorem~\ref{theor: MAP}. The version of Proposition \ref{prop:capsula convexa} for $\mathcal{G}(B_X)$ requires the addition of $\delta(0)$ to both considered sets. This addition has  impact in Corollary \ref{cor: compacto en G0} and Theorem \ref{teo: top tau delta}, which in turn affects the proofs of Propositions \ref{prop: equivalent topologies} and \ref{prop: densidad AP} and Theorem \ref{teo: equivalencia AP}. All these results are valid for $\mathcal{G}(B_X)$ after the mentioned modifications. Alternatively, this also follows from the fact that $\mathcal G(B_X)$ is isometric to a $1$-complemented subspace of $\mathcal G_0(B_{X\oplus_1\mathbb C})$ (just note that the map $\Phi$ in Proposition \ref{prop: G vs G_0} is the adjoint of the linearization $T_{F}$ of the map $F(x,\lambda)=\delta(x)+(\lambda-1)\delta(0)$).
	Also note that the square diagram~\eqref{diagrama cuadrado} can be made for $\mathcal{G}(B_X)$ but there is no equality between the norms of $T_{\delta_Y\circ f}$ and $f$.
\end{remark}

\section{Relation between $\mathcal{G}_0(B_X)$ and $\mathcal{G}_0(B_Y)$ when $X\subset Y$} \label{section: relation}

Recall that, given metric spaces $M, N$ with $0\in M\subset N$, the (real) Lipschitz-free space $\mathcal F(M)$ canonically identifies with a subspace of $\mathcal F(N)$. This follows from the McShane extension theorem asserting that for every $f\in \Lip_0(M, \mathbb R)$ there is $\tilde{f}\in \Lip_0(N, \mathbb R)$ with $\tilde{f}|_M=f$ and $L(f)=L(\tilde{f})$, see e.g. \cite[Th. 1.33]{Weaver}. Note in passing that in the complex-valued case all extensions can have a larger Lipschitz constant.
This is why our next goal is to analyze the corresponding relation between $\mathcal G_0(B_X)$ and $\mathcal G_0(B_Y)$ when $X\subset Y$. Then $B_X\subset B_Y$ and the restriction mapping has norm one:
\begin{align*}
	\mathcal HL_0(B_Y) & \to \mathcal HL_0(B_X)\\
	f & \mapsto f|_{B_X}.
\end{align*}

Then, the following mapping also has norm one:
\begin{align*}
	\rho:  \mathcal G_0(B_X) & \to \mathcal G_0(B_Y)\\
	\varphi & \mapsto \widehat{\varphi},
\end{align*} where $\widehat{\varphi}(f)=\varphi (f|_{B_X})$.

Whenever $\rho$ is an isometry,  
we write $\mathcal G_0(B_X)\subset \mathcal G_0(B_Y)$. Then, by the Hahn-Banach theorem, every element of $\mathcal HL_0(B_X)$ would have a norm preserving extension to $\mathcal HL_0(B_Y)$. Since there exist polynomials which cannot be extended to a larger space it is not always true that $\mathcal{G}_0(B_X)\subset \mathcal{G}_0(B_Y)$. Moreover, the previous argument can be clearly reversed, so: $\mathcal{G}_0(B_X)\subset \mathcal{G}_0(B_Y)$ if and only if every  $f\in \mathcal HL_0(B_X)$  has a norm preserving extension to $\mathcal HL_0(B_Y)$.

We study some cases where this norm preserving extension occurs. All are cases where we have an extension morphism. The simplest occurs when $X$ is 1-complemented in $Y$. Here, the complementation also spreads to  $\mathcal{G}_0(B_X)$.

\begin{proposition}\label{prop:1comp}
	If $X$ is 1-complemented in $Y$ then $\rho$ is an isometry and $\mathcal{G}_0(B_X)$ is a 1-complemented subspace of $\mathcal{G}_0(B_Y)$.
\end{proposition}

\begin{proof}
	Let $\pi\colon Y\to X$ be a norm-one projection. Given $f\in \mathcal HL_0(B_X)$ the mapping $f\circ\pi$ belongs to $\mathcal HL_0(B_Y)$ with $L(f\circ\pi)\le L(f)$ and $(f\circ\pi)|_{B_X}=f$. Now, for each $\varphi\in\mathcal{G}_0(B_X)$,
	$$
	\|\varphi\|=\sup_{f\in B_{\mathcal HL_0(B_X)}} |\varphi(f)|  =  \sup_{f\in B_{\mathcal HL_0(B_X)}} |\widehat{\varphi}(f\circ\pi)|\le\|\widehat\varphi\|.
	$$ Thus, $\|\varphi\|=\|\widehat\varphi\|$, meaning that $\rho$ is an isometry. Finally, we derive that $\mathcal{G}_0(B_X)$ is 1-complemented in $\mathcal{G}_0(B_Y)$ through the following projection:
	\begin{align*}
		\mathcal G_0(B_Y) & \to \mathcal G_0(B_X)\\
		\psi & \mapsto [f\mapsto \psi(f\circ\pi)].
	\end{align*}
\end{proof}

M. Jung has proved recently that $\mathcal G^\infty(B_X)$ does not have  the Radon-Nikodym property (RNP) for any $X$ \cite{Jung}. Here we obtain the same result for $\mathcal G_0(B_X)$.

\begin{corollary} \label{cor:RNP}
	The space $\mathcal G_0(B_X)$ fails to have the Radon-Nikodym Property for every complex Banach space $X$.
\end{corollary}

\begin{proof}
	The space $\mathcal G^\infty(\mathbb D)$ fails to have the RNP since its the unit ball does not have extreme points \cite{Ando}. Thus, by the isometry presented in Proposition \ref{prop: quotient}, the same holds for $\mathcal G_0(\mathbb D)$. Since $\mathbb C$ is $1$-complemented in $X$, Proposition \ref{prop:1comp} yields that $\mathcal  G_0(\mathbb D)$ is a subspace of $\mathcal G_0(B_X)$ and we are done.
\end{proof}

Another situation when we have an extension morphism is when $Y=X^{**}$. Recall that, given $f\in \Hinf(B_X)$, we can consider its AB extension $\tilde{f}\in \Hinf(B_{X^{**}})$ \cite{AB}. The AB extension, which defines an isometry from $\Hinf(B_X)$ to $\Hinf(B_{X^{**}})$ \cite{DG}, is a topic widely developed in the literature. For instance, it is essential in the description of the spectrum (or maximal ideal space) of the Banach algebra $\Hinf(B_X)$. Another ingredient that usually appears associated with the AB extension and its properties is the notion of \emph{symmetrically regular space}. Both these concepts  have their origin in the study initiated by Arens \cite{Arens1, Arens2} about extending the product of a Banach algebra to its bidual.

For an $n$-linear mapping $A:X\times\cdots\times X\to Y$ the canonical extension $\widetilde A:X^{**}\times\cdots\times X^{**}\to Y^{**}$ is given by consecutive weak-star convergence in the following way:
$$
\widetilde A(x^{**}_1,\dots, x^{**}_n)(y^{*})=
\lim_{\alpha_1}\dots \lim_{\alpha_n} y^*(A(x_{\alpha_1},\dots, x_{\alpha_n}))
$$
where each $(x_{\alpha_i})\subset X$ is a net which is weak-star convergent to $x_i^{**}$ and $y^{*}\in Y^{*}$. Now, the AB extension of a homogeneous polynomial $P\in\mathcal P(^nX,Y)$ is given by $\widetilde P\in\mathcal P(^nX^{**},Y^{**})$ which is  defined, for $x^{**}\in X^{**}$, in the expected way:
$$
\widetilde P(x^{**})=\widetilde{\widecheck{P}}(x^{**},\dots, x^{**}).
$$
This  provides a way to extend bounded holomorphic functions $f\in\mathcal H^\infty(B_X,Y)\leadsto \widetilde f\in \mathcal H^\infty(B_{X^{**}},Y^{**})$ and  we know from \cite{DG} that this extension is an isometry: $\|f\|=\|\widetilde f\|$.

Recall that
$X$ is said to be {\em regular} if every continuous bilinear mapping $A : X \times X \to \mathbb{C}$ is Arens regular.
That is, the following two extensions of $A$  to $X^{**} \times X^{**} \to \mathbb{C}$ coincide:
$$\lim_\alpha \lim_\beta A(x_\alpha,y_\beta)\quad \textrm{ and } \quad
\lim_\beta \lim_\alpha A(x_{\alpha},y_{\beta}),$$
where $(x_\alpha)$ and $(y_\beta)$ are nets in $X$ converging weak-star to points $x_0^{**}$  and $y_0^{**}$
in $X^{**}$. The space
$X$ is {\em symmetrically regular} if the above holds for every continuous symmetric bilinear form. Equivalently, $X$ is (symmetrically) regular if any continuous (symmetric) linear mapping $T:X\to X^*$ is weakly compact.
Several equivalent characterizations of this notion can be seen in \cite[Th. 8.3]{acg} and some interesting properties appeared in \cite[Section 1]{AGGM}. As examples of non reflexive regular (and hence, symmetrically regular) Banach spaces we have, for instance, those that satisfy property (V) of Pe{\l}czy\'{n}ski, like $c_0$, $C(K)$ or $\mathcal H^\infty(\D)$ while typical non symmetrically regular spaces are $\ell_1$ and $X\oplus X^*$, for any non reflexive space $X$.
Also,  Leung \cite[Th. 12]{Leung} provided an example of a symmetrically regular space that is not regular and in \cite{AGGM} it is showed that $c_0(\ell_1^n)$ is regular but its bidual $\ell_\infty(\ell_1^n)$ is not symmetrically regular.

We now want to work with the AB extension for elements in $\mathcal HL_0(B_X)$. For $f\in \mathcal HL_0(B_X)$, in order to compute the Lipschitz constant of $\tilde{f}$ we need to deal with the differential of the AB extension, $d\widetilde{f}$ which belongs to $\mathcal H(B_{X^{**}}, X^{***})$.  Instead, we do know the norm of the AB extension of the differential $\reallywidetilde{df}\in\Hinf (B_{X^{**}}, X^{***})$. Fortunately, on symmetrically regular spaces they coincide:

\begin{proposition}
	If $X$ is symmetrically regular and $f\in \mathcal HL_0(B_X)$ then $d\widetilde{f}=\reallywidetilde{df}$.
\end{proposition}

\begin{proof}
	If  $f=\sum_{k=0}^\infty P^kf(0)$ then the series expansion of $df$ at $0$ is given by $df=\sum_{k=0}^\infty dP^kf(0)$. Thus, $\reallywidetilde{df}=\sum_{k=0}^\infty \reallywidetilde{(dP^kf(0))}$. On the other hand, $\widetilde{f}=\sum_{k=0}^\infty \widetilde{P^kf(0)}$ and so $d\widetilde{f}=\sum_{k=0}^\infty d(\widetilde{P^kf(0)})$.
	
	Therefore, the result is proved once we show that for any given $m\in \mathbb N$ and any $P\in\mathcal P(^mX)$,  $\widetilde{dP}=d\widetilde{P}$. Note that in this case $\widetilde P\in\mathcal P(^mX^{**})$, $dP\in\mathcal P(^{m-1}X,X^*)$ while both $\widetilde{dP}$ and $d\widetilde P$ belong to $\mathcal P(^{m-1}X^{**},X^{***})$.
	
	When $X$ is symmetrically regular, it follows from \cite[Th. 8.3]{acg} that $\widetilde{\widecheck{P}}= \widecheck{\widetilde{P}}$.  The argument is now complete because, for each $x^{**}, y^{**}\in X^{**}$ we have
	$\widetilde{dP}(x^{**})(y^{**})= m \widetilde{\widecheck{P}} (x^{**},\dots, x^{**}, y^{**})$ and $d\widetilde{P}(x^{**})(y^{**})= m \widecheck{\widetilde{P}} (x^{**},\dots, x^{**}, y^{**})$.
\end{proof}

\begin{proposition} \label{prop:symmetrically regular}
	If $X$ is symmetrically regular then the AB extension mapping
	\begin{align*}
		E:  \mathcal HL_0(B_X) & \to \mathcal HL_0(B_{X^{**}})\\
		f & \mapsto \widetilde{f}
	\end{align*}  is an isometry.
\end{proposition}

\begin{proof}
	If $f\in \mathcal HL_0(B_X)$ then its norm is given by $\|df\|$. By \cite{DG}, $\|df\|=\|\reallywidetilde{df}\|$.
	Also, by the previous proposition we know that $d\widetilde{f}=\reallywidetilde{df}$. So, we obtain that $\|df\|=\|d\widetilde{f}\|$, meaning that $\widetilde f$ does indeed belong to $\mathcal HL_0(B_{X^{**}})$ and that the mapping $f\mapsto \widetilde f$ is an isometry.
\end{proof}

In the previous result  symmetric regularity is used to obtain that $d\widetilde{f}=\reallywidetilde{df}$. Actually we only need the identity of their norms: $\|d\widetilde{f}\|=\|\reallywidetilde{df}\|$. We do not know if this equality holds in general.

\begin{corollary}\label{cor:symmetrically regular}
	If $X$ is symmetrically regular then $\mathcal{G}_0(B_X)\subset \mathcal{G}_0(B_{X^{**}})$.
\end{corollary}

Note that in the above corollary the hypothesis of symmetric regularity is not a necessary condition since, for example, for $X=\ell_1$ the result holds due to Proposition  \ref{prop:1comp}.

A generalization of this procedure (which, however, uses the AB extension in its definition) is when there exists an isometric extension morphism $s\colon X^*\to Y^*$. This happens, for instance, when $X$ is an M-ideal in $Y$. More generally, if $X\subset Y$ then the existence of an isometric extension morphism $s\colon X^*\to Y^*$ is equivalent to $X^{**}$ being 1-complemented in $Y^{**}$. Actually, the existence of an isometric extension morphism $s\colon X^*\to Y^*$ is equivalent to $X$ being 1-locally complemented in $Y$ (see the definition in the next section and the comment before Corollary \ref{cor:1-complemented}). 

Note that $s(x^*)(x)=x^*(x)$ for all $x\in X$, $x^*\in X^*$ and that $\|s(x^*)\|=\|x^*\|$. This extension transfers to $\Hinf(B_X)$ in the following way:
\begin{align*}
	\overline{s}:  \Hinf(B_X) & \to \Hinf(B_{Y})\\
	f & \mapsto \widetilde{f}\circ s^*\circ i_Y,
\end{align*} where $i_Y:Y\to Y^{**}$ is the canonical inclusion.

The mapping $\overline{s}$ is an isometric extension from $ \Hinf(B_X)$ to $\Hinf(B_{X^{**}})$. Again, to work in $\mathcal HL_0(B_X)$ we require a symmetrically regular hypothesis.

\begin{proposition}\label{prop: extension s}
	If $X$ is symmetrically regular, $X\subset Y$ and there is an isometric extension morphism $s:X^*\to Y^*$  then
	\begin{align*}
		\overline{s}:  \mathcal HL_0(B_X) & \to \mathcal HL_0(B_{Y})\\
		f & \mapsto \widetilde{f}\circ s^*\circ i_Y
	\end{align*}  is an isometric extension.
\end{proposition}

\begin{proof}
	For any $P\in\mathcal P(^mX)$ we have that $\overline{s}(P)\in \mathcal P(^mY)$ and $d(\overline{s}(P))\in\mathcal P(^{m-1}Y, Y^*)$. Now, for $y,z\in B_Y$,
	\begin{align*}
		d(\overline{s}(P))(y)(z) &= m \widecheck{(\overline{s}(P))}(y,\dots, y, z) = m \widecheck{\widetilde P}\left(s^*(i_Y(y)),\dots, s^*(i_Y(y)),s^*(i_Y(z))\right)\\
		&=d\widetilde{P} (s^*(i_Y(y)))(s^*(i_Y(z))) = (i_Y^*\circ s^{**}\circ d\widetilde{P}\circ s^*\circ i_Y)(y)(z).
	\end{align*}
	
	This says that $d(\overline{s}(P))=i_Y^*\circ s^{**}\circ d\widetilde{P}\circ s^*\circ i_Y$ for every polynomial $P\in\mathcal P(^mX)$. Then, the same equality holds for every $f\in\mathcal HL_0(B_X)$:
	$$
	d(\overline{s}(f))=i_Y^*\circ s^{**}\circ d\widetilde{f}\circ s^*\circ i_Y.
	$$ Since $X$ is symmetrically regular, by Proposition \ref{prop:symmetrically regular} we obtain that $\|d(\overline{s}(f))\|\le \|d\widetilde{f}\|=\|df\|$. Also, note that for $x\in B_X$, we have $ s^*\circ i_Y(x)=i_X(x)$. This implies that $d\widetilde{f}(s^*(i_Y(x))=i_{X^*}(df(x))$. Therefore,
	$$
	d(\overline{s}(f))(x)=i_Y^*\circ s^{**}\left(i_{X^*}(df(x))\right) = s(df(x)).
	$$ This equality and the fact that $s$ is an isometry allow us to derive the other inequality:
	\begin{align*}
		\|d(\overline{s}(f))\| & \ge \sup_{x\in B_X} \|d(\overline{s}(f))(x)\| = \sup_{x\in B_X} \|s(df(x))\| \\
		&= \sup_{x\in B_X} \|df(x)\|=\|df\|,
	\end{align*}
	which concludes the proof.
\end{proof}

\begin{corollary}\label{cor:GXsubsetGY}
	If $X$ is symmetrically regular, $X\subset Y$ and there is an isometric extension morphism $s\colon X^*\to Y^*$  then $\mathcal{G}_0(B_X)\subset \mathcal{G}_0(B_Y)$.
\end{corollary}

\subsection{Dual isometric spaces}
It is known that there exist non isomorphic Banach spaces with isomorphic duals. Attending to that, D\'{\i}az and Dineen \cite{DiDi} posed the following question: if $X$ and $Y$ are Banach spaces such that $X^*$ and $Y^*$ are isomorphic, under which conditions is it true that $\mathcal P(^nX)$ and $\mathcal P(^nY)$ are isomorphic for every $n\ge 1$? That is, if $X^*$ and $Y^*$ are isomorphic (i. e. the spaces of 1-homogeneous polynomials are isomorphic) does it imply that the spaces of $n$-homogeneous polynomials are isomorphic for every $n$? They also gave a partial answer to this question. Later, a relaxation of the conditions was obtained by Cabello-S\'{a}nchez, Castillo and Garc\'{\i}a \cite[Th. 1]{CCG} and Lassalle and Zalduendo \cite[Th. 4]{LZ} independently, proving that the answer is affirmative whenever $X$ and $Y$ are symmetrically regular. We present here a version of this result for holomorphic Lipschitz functions on the ball. Since we need to remain inside the ball when changing the space we have to restrict ourselves to the case of isometric  isomorphisms.

\begin{proposition}
	If $X$ and $Y$ are symmetrically regular Banach spaces such that $X^*$ and $Y^*$ are isometrically isomorphic then $\mathcal HL_0(B_X)$ and $\mathcal HL_0(B_Y)$ are isometrically isomorphic as well.
\end{proposition}

\begin{proof}
	Let us denote by $s\colon X^*\to Y^*$ the isometric isomorphism and
	consider the mapping  $\overline{s}:  \mathcal HL_0(B_X)  \to \mathcal HL_0(B_{Y})$ as in Proposition \ref{prop: extension s}. By the proof of that proposition we derive that $\overline s$ is continuous and $\|\overline s\|\le 1$. Since $Y$ is symmetrically regular, we can use the same procedure for the mapping $\overline{s^{-1}}:  \mathcal HL_0(B_Y)  \to \mathcal HL_0(B_{X})$ leading to $\|\overline{s^{-1}}\|\le 1$. Finally, appealing to \cite[Cor. 3]{LZ} we obtain that $\overline{s^{-1}}\circ \overline s(P)=P$ for every homogeneous polynomial $P$ on $X$ and, hence,    $\overline{s^{-1}}\circ \overline s(f)=f$ for every $f\in\mathcal HL_0(B_X)$. Indeed, if $\sum_{k=0}^\infty P^k$ is the Taylor series expansion of a given $f\in\mathcal HL_0(B_X)$, then $\tilde f(z)=\sum_{k=0}^\infty \tilde P^k(z)$ for every $z\in B_{X^{**}}$. Thus
	\[
	\overline s(f)(y)=\tilde f(s^*(i_Y(y)))=\sum_{k=0}^\infty \tilde P^k(s^*(y))=\sum_{k=0}^\infty \overline s(P^k)(y),
	\]
	for every $y \in Y$. From here
	\begin{eqnarray*}
		\overline{s^{-1}}(\overline s(f))(x) & = &\widetilde{\overline s(f)}\big((s^{-1})^* (i_X(x))\big)=\sum_{k=0}^\infty \widetilde{\overline{s}(P^k)}\big((s^{-1})^* (i_X(x)\big)\\ & = & \sum_{k=0}^\infty \overline{ s^{-1}}(\overline s(P^k))(x)=\sum_{k=0}^\infty P^k(x)=f(x),
	\end{eqnarray*}
	for every $x \in X$. Analogously one can check that $\overline{s}\circ\overline{s^{-1}}(f)=f$ for every $f\in \mathcal HL(B_Y)$.
\end{proof}

In the previous proposition we can change the hypothesis of $X$ and $Y$ being symmetrically regular by $X$ or $Y$ being regular. Indeed, it is proved in \cite[Rmk. 2]{LZ} (see also \cite[Prop. 1]{CCG}) that if $X^*$ and $Y^*$ are  isomorphic and $X$ is regular then so is $Y$.

\subsection{Mapping between $\mathcal{G}_0(B_X)$ and $\mathcal{G}_0(B_Y)$}

Any linear mapping between $X$ and $Y$ produces a mapping between $\mathcal{G}_0(B_X)$ and $\mathcal{G}_0(B_Y)$ by a canonical procedure (actually, two canonical procedures depending on the norm of the mapping).

(i) Let $\psi\colon X\to Y$ a linear mapping with $\|\psi\|\le 1$. Note that $L(\psi)=\|\psi\|$ in this case. Since $\psi(B_X)\subset B_Y$ we can define the canonical mapping with norm $\le 1$:
\begin{align*}
	\mathcal HL_0(B_Y) & \to \mathcal HL_0(B_X)\\
	f & \mapsto f\circ\psi.
\end{align*}

Thus, the following also has  norm $\le 1$:

\begin{align*}  T_{\delta_Y\circ \psi}\colon  \mathcal G_0(B_X) & \to \mathcal G_0(B_Y)\\
	\varphi & \mapsto \widehat{\varphi},
\end{align*} where $\widehat{\varphi}(f)=\varphi (f\circ\psi)$.

(ii) When $\|\psi\|>1$ the previous construction does not work but we can appeal to a linearization plus differentiation process (as we used to show that $X$ is a 1-complemented subspace of $\mathcal{G}_0(B_X)$).

Let $\psi\in \mathcal L(X,Y)$ so that $\psi|_{B_X}\in\mathcal HL_0(B_X,Y)$. We have the usual commutative diagram:

\begin{equation*}
	\xymatrix{
		B_X \ar[r]^{\psi|_{B_X}}  \ar[d]_{\delta_X}    &  Y \ar[d]^{d\delta_Y (0)} \\
		\mathcal{G}_0(B_X)  \ar[ru]_{T_{\psi}}  & \mathcal{G}_0(B_Y)
	}
\end{equation*} where $T_{\psi}\in\mathcal L(\mathcal{G}_0(B_X),Y)$.

Applying the differential at 0 to the equality $\psi|_{B_X}=T_{\psi}\circ \delta_X$ we get the commutative diagram:
\begin{equation*}
	\xymatrix{
		X \ar[r]^{\psi}  \ar[d]_{d\delta_X (0)}    &  Y \ar[d]^{d\delta_Y (0)} \\
		\mathcal{G}_0(B_X) \ar[r]_{d\delta_Y (0)\circ T_{\psi}} \ar[ru]_{T_{\psi}}  & \mathcal{G}_0(B_Y).
	}
\end{equation*}

Note that the linear mapping $d\delta_Y (0)\circ T_{\psi}:\mathcal{G}_0(B_X)\to \mathcal{G}_0(B_Y)$ has norm less than or equal to $\|\psi\|$.

\section{Local complementation in the bidual}\label{Section: local complementation}

In this section, we are interested in the relationship between $\mathcal G_0(B_{X^{**}})$ and $\mathcal G_0(B_X)^{**}$ under the hypothesis of $X^{**}$ having the MAP, in the spirit of what is done in \cite{CG}.

We begin with a result about a special approximation behavior in the case that the bidual space has the MAP.

\begin{proposition}
	\label{prop:aproximacion-bidual-map}
	Let $X,Y$ be Banach spaces such that $X^{**}$ has the MAP. For each $f\in\mathcal HL_0(B_{X^{**}}, Y)$ with $L(f)=1$ there exists a net $(Q_\alpha)\subset \mathcal P_{f,0}(X,Y)$ with $L(Q_\alpha|_{B_X})\le 1$ satisfying $\widetilde{Q}_\alpha(x^{**})\to f(x^{**})$ for all $x^{**}\in B_{X^{**}}$.
\end{proposition}

\begin{proof}
	By Proposition \ref{Prop:bola unidad} it is enough to consider $f=P\in \mathcal P_0(X^{**},Y)$ with $L(P|_{B_{X^{**}}})\le 1$. If $X^{**}$ has the MAP  we can appeal to \cite[Cor. 1]{CG} to obtain a net of finite rank mappings $(t_\alpha)\subset \mathcal L(X,X^{**})$ with $\|t_\alpha\|\le 1$ and $t_\alpha^{**}(x^{**})\to x^{**}$ for all $x^{**}\in X^{**}$. Now we define $Q_\alpha=P\circ t_\alpha$, which clearly belongs to $\mathcal P_{f,0}(X,Y)$. Note that, for any $x,y\in B_X$,
	$$
	\|Q_\alpha(x)-Q_\alpha(y)\|=\|P(t_\alpha(x))-P(t_\alpha(y))\|\le L(P|_{B_{X^{**}}}) \|t_\alpha\| \|x-y\|\le  \|x-y\|.
	$$
	Then, $L(Q_\alpha|_{B_X})\le 1$. Since $t_\alpha$ is a finite rank mapping, we have that $t_\alpha^{**}\in\mathcal L(X^{**},X^{**})$. Hence, $\widetilde Q_\alpha=\widetilde P\circ t_\alpha^{**}=P\circ t_\alpha^{**}$. As a consequence, $\widetilde Q_\alpha(x^{**})=P(t_\alpha^{**}(x^{**}))\to P(x^{**})$ for all $x^{**}\in B_{X^{**}}$.
	
\end{proof}

For a symmetrically regular space $X$,  we consider the following mapping
\begin{align*}
	\Theta:  B_{X^{**}} & \to \mathcal G_0(B_X)^{**}=\mathcal HL_0(B_X)^*\\
	x^{**} & \mapsto [f\in\mathcal HL_0(B_X)\mapsto \widetilde f(x^{**})].
\end{align*}
\begin{proposition}
	If $X$ is symmetrically regular then $\Theta$ belongs to $\mathcal HL_0(B_{X^{**}}, \mathcal G_0(B_X)^{**})$ with $L(\Theta)=1$.
\end{proposition}

\begin{proof}
	If $X$ is symmetrically regular, by Proposition \ref{prop:symmetrically regular}, the AB extension is an isometry from $\mathcal HL_0(B_X)$ into $\mathcal HL_0(B_{X^{**}})$, so $\Theta $ is well defined.
	For any $f\in \mathcal HL_0(B_X)$, we have $\Theta(\cdot)(f)=\widetilde f$, meaning that $\Theta$ is weak-star holomorphic and thus, it is holomorphic. Also, $\Theta(0)=0$ and for any $x^{**}, y^{**}\in B_{X^{**}}$, once again by the symmetric regularity of $X$ we have
	$$
	\|\Theta(x^{**})-\Theta(y^{**})\| = \sup_{f\in B_{\mathcal HL_0(B_X)}} \|\widetilde f(x^{**})-\widetilde f(y^{**})\| \le \|x^{**}-y^{**}\|.
	$$
	This means that $\Theta\in\mathcal HL_0(B_{X^{**}}, \mathcal G_0(B_X)^{**})$ with $L(\Theta)\le 1$. On the other hand,
	$$
	\|\Theta(x^{**})-\Theta(y^{**})\|\ge \sup_{x^*\in B_{X^*}} |x^{**}(x^*)-y^{**}(x^*)|=\|x^{**}-y^{**}\|.
	$$ Therefore, $L(\Theta)=1$.
\end{proof}

As a consequence of the previous proposition, if $X$ is symmetrically regular we can linearize the mapping $\Theta$:
\begin{equation*}
	\xymatrix{
		B_{X^{**}} \ar[r]^{\Theta\ }  \ar[d]_{\delta_{X^{**}}}    &  \mathcal G_0(B_X)^{**}  \\
		\mathcal G_0(B_{X^{**}})  \ar[ru]_{T_{\Theta}}  &
	}
\end{equation*}
This produces a linear mapping $T_\Theta\in\mathcal L(\mathcal G_0(B_{X^{**}}), \mathcal G_0(B_X)^{**})$ with $\|T_\Theta\|=L(\Theta)=1$.

Motivated by the Principle of Local Reflexivity, Kalton \cite{Ka} introduced the following definition:

\begin{definition}
	Given Banach spaces $X\subset Y$ we say that $X$ is 1-locally complemented in $Y$ if for every $\varepsilon>0$ and every finite dimensional subspace $F$ of $Y$ there exist a linear mapping $T:F\to X$ such that $\|T\|\le 1+\varepsilon$ and $T(x)=x$ for all $x\in F\cap X$.
\end{definition}

Note that the Principle of Local Reflexivity says that $X$ is 1-locally complemented in $X^{**}$, for any Banach space $X$.

\begin{theorem}
	If $X$ is symmetrically regular and $X^{**}$ has the MAP then $T_\Theta$ embeds $\mathcal G_0(B_{X^{**}})$ as a 1-locally complemented subspace of $\mathcal G_0(B_X)^{**}$. In particular,  $T_\Theta$ is an isometry.
\end{theorem}

\begin{proof}
	We know that the mapping $\delta_{X^{**}}$ belongs to $\mathcal HL_0(B_{X^{**}}, \mathcal G_0(B_{X^{**}}))$ with $L(\delta_{X^{**}})=1$. Thus, we can apply Proposition \ref{prop:aproximacion-bidual-map} to get a net $(Q_\alpha)\subset \mathcal P_{f,0}(X,\mathcal G_0(B_{X^{**}}))$ with $L(Q_\alpha|_{B_X})\le 1$ such that $\widetilde{Q}_\alpha(x^{**})\to \delta_{X^{**}}(x^{**})$ for all $x^{**}\in B_{X^{**}}$.
	
	Consider the following two commutative diagrams:
	\begin{equation*}
		\xymatrix{
			B_{X} \ar[r]^{Q_\alpha|_{B_X} }  \ar[d]_{\delta_{X}}    &  \mathcal G_0(B_{X^{**}})  \\
			\mathcal G_0(B_{X})  \ar[ru]_{T_{Q_\alpha}}  &
		} \qquad\qquad \xymatrix{
			B_{X^{**}} \ar[r]^{\widetilde Q_\alpha|_{B_{X^{**}} } }  \ar[d]_{\delta_{X^{**}}}    &  \mathcal G_0(B_{X^{**}})  \\
			\mathcal G_0(B_{X^{**}})  \ar[ru]_{T_{\widetilde Q_\alpha} } &
		}
	\end{equation*}
	Note that, since $X$ is symmetrically regular we have
	$$
	\|T_{Q_\alpha}\|=L(Q_\alpha|_{B_X} )= L(\widetilde Q_\alpha|_{B_X} ) = \|T_{\widetilde Q_\alpha}\| \le 1.
	$$ For each $\alpha$, since $T_{Q_\alpha}$ is a finite rank operator we have that $T_{Q_\alpha}^{**}$ belongs to $\mathcal L(\mathcal G_0(B_{X})^{**},\mathcal G_0(B_{X^{**}}))$. Thus, we have the following diagram
	
	\begin{eqnarray*}
		\xymatrix{{ \mathcal G_0(B_{X^{**}})} \ar[rr]^{T_\Theta} \ar[dr]_{T_{\widetilde Q_\alpha}} & & {\mathcal G_0(B_{X})^{**} }  \ar[ld]^{T_{ Q_\alpha}^{**}} \\
			&{  \mathcal G_0(B_{X^{**}})} &  }
	\end{eqnarray*}
	The space $\mathcal G_0(B_{X^{**}})$ has the MAP witnessed by the net $(T_{\widetilde Q_\alpha})$ thanks to (the proof of) Theorem~\ref{theor: MAP}. Appealing to \cite[Lem. 4]{CG}, the proof will be completed once we check that the previous diagram is commutative. For this, it is enough to prove that $T_{\widetilde Q_\alpha}(\delta_{X^{**}}(x^{**}))=T_{Q_\alpha}^{**}\circ T_\Theta(\delta_{X^{**}}(x^{**}))$ for every $x^{**}\in B_{X^{**}}$.
	
	On the one hand we know that $T_{\widetilde Q_\alpha}(\delta_{X^{**}}(x^{**}))=\widetilde Q_\alpha
	(x^{**})$. On the other hand, $T_{Q_\alpha}^{**}\circ T_\Theta(\delta_{X^{**}}(x^{**}))=T_{Q_\alpha}^{**}(\Theta(x^{**}))$. To understand this element of $\mathcal G_0(B_{X^{**}})$ let us see how it acts on any $f\in \mathcal HL_0(B_{X^{**}})$:
	
	\begin{equation} \label{eqn:loc complem}
		\langle T_{Q_\alpha}^{**}(\Theta(x^{**})),f\rangle = \langle \Theta(x^{**}), T_{Q_\alpha}^{*}(f)\rangle.
	\end{equation}

	Now, $T_{Q_\alpha}^{*}(f)$ belongs to $\mathcal HL_0(B_{X})$ and for any $x\in B_X$ satisfies
	$$
	T_{Q_\alpha}^{*}(f)(x)= \langle T_{Q_\alpha}^{*}(f),\delta_X(x)\rangle = \langle f, T_{Q_\alpha}(\delta_X(x))\rangle = \langle f, Q_\alpha(x)\rangle = (T_f\circ Q_\alpha) (x).
	$$
	Then, $T_{Q_\alpha}^{*}(f)=T_f\circ Q_\alpha$. Replacing  this equality in \eqref{eqn:loc complem} and using the definition of $\Theta$ and the fact that the range of $\widetilde Q_\alpha$ is contained in $\mathcal G_0(B_{X^{**}})$ we derive
	\begin{align*}
		\langle T_{Q_\alpha}^{**}(\Theta(x^{**})),f\rangle  & =   \langle \Theta(x^{**}), T_f\circ Q_\alpha\rangle = \widetilde{T_f\circ Q}_\alpha(x^{**}) = T_f^{**}\circ \widetilde Q_\alpha(x^{**})\\
		&= T_f(\widetilde Q_\alpha(x^{**})) = \langle \widetilde Q_\alpha(x^{**}), f\rangle, \quad \textrm{for all }f\in \mathcal HL_0(B_{X^{**}}).
	\end{align*}
	
	Therefore, $T_{Q_\alpha}^{**}(\Theta(x^{**}))=\widetilde Q_\alpha(x^{**})$ and thus $T_{Q_\alpha}^{**}\circ T_\Theta(\delta_{X^{**}}(x^{**}))=T_{\widetilde Q_\alpha}(\delta_{X^{**}}(x^{**}))$
	for every $x^{**}\in B_{X^{**}}$, which finishes the proof.

\end{proof}

It is known (see, for instance, \cite[Lem. 3]{CG} or \cite[Th. 3.5]{Ka}) that $X$ is 1-locally complemented in $Y$ if and only if $X^*$ is 1-complemented in $Y^*$ (with projection the restriction mapping). This is also equivalent to $X^{**}$ being 1-complemented in $Y^{**}$ (under the natural embedding).

\begin{corollary}\label{cor:1-complemented}
	If $X$ is symmetrically regular and $X^{**}$ has the MAP then $\mathcal HL_0(B_{X^{**}})$ is isometric to a 1-complemented subspace of $\mathcal HL_0(B_{X})^{**}$.
\end{corollary}

Under the same conditions of the previous results we can also obtained a version for holomorphic Lipschitz functions of the following characterization of unique norm preserving extensions to the bidual proved by Godefroy in \cite{Go}.

\begin{lemma}
	Let $X$ be a Banach space and $x^{* }\in X^*$ with $\|x^*\|=1$.  The following are equivalent:
	\begin{itemize}
		\item[$(i)$]  $x^{\ast }$ has a unique norm preserving extension to a  functional on $X^{\ast \ast }$.
		\item[$(ii)$] The function $Id_{\overline B_{X^{\ast }}}: (\overline B_{X^{\ast }} ,w^{\ast }) \longrightarrow (\overline B_{X^{\ast }} ,w) $
		is continuous at $x^{\ast }$.
	\end{itemize}
\end{lemma}

Aron, Boyd and Choi \cite{ABC} gave a version of this result for homogeneous polynomials. Later, other extensions appeared (for instance, in \cite{DGG} for ideals of homogeneous polynomials and in \cite{DiFe} for bilinear mappings in operator spaces).

Now, the statement of the theorem in our setting is the following:

\begin{theorem}
	Suppose $X$ is symmetrically regular and  $X^{**}$ has the MAP. Consider a function
	$f\in\mathcal HL_0(B_X)$ with
	$L(f)=1$. Then, the following are
	equivalent:
	\begin{enumerate}
		\item[$(i)$] $f$ has a unique norm preserving extension to  $\mathcal HL_0(B_{X^{**}})$.
		\item[$(ii)$] The AB extension from $(\overline B_{\mathcal HL_0(B_X)}, w^*)$ to $(\overline B_{\mathcal HL_0(B_{X^{**}})}, w^*)$ is continuous at
		$f$.
		\item[$(iii)$] If the net $(f_{\alpha})\subset  \overline B_{\mathcal HL_0(B_X)}$ converges pointwise to $f$,
		then $(\widetilde f_\alpha)\subset  \overline B_{\mathcal HL_0(B_{X^{**}})}$ converges pointwise to
		$\widetilde f$.
	\end{enumerate}
\end{theorem}

\begin{proof}
	$(i)\Rightarrow (ii)$ Let $(f_{\alpha})\subset  \overline B_{\mathcal HL_0(B_X)}$ be a net weak-star convergent to a function $f\in  \overline B_{\mathcal HL_0(B_X)}$. By the weak-star compactness of the ball $\overline{B}_{\mathcal HL_0(B_{X^{**}})}$ there is a subnet $(\widetilde f_\beta)$ weak-star convergent to a function $g\in \overline B_{\mathcal HL_0(B_{X^{**}})}$. Since for each $x\in B_X$, $\widetilde f_\alpha (x)=f_\alpha (x)\to f(x)$ we derive that $g|_{B_X}=f$. Also, since $L(g)\le 1=L(f)$, it follows that $L(g)=L(f)$, which means that $g$ is a norm preserving extension of $f$. By $(i)$ and Proposition \ref{prop:symmetrically regular} we obtain that $g=\widetilde f$. Now, a standard subnet argument shows that the whole net $(\widetilde f_\alpha)$ must converge weak-star to $\widetilde f$.
	
	$(ii)\Rightarrow (iii)$ It is clear due to Proposition \ref{prop:diagram} $(d)$.
	
	$(iii)\Rightarrow (i)$ Let $g\in \overline B_{\mathcal HL_0(B_{X^{**}})}$ be a norm preserving extension of $f$. By Proposition \ref{prop:aproximacion-bidual-map} there is a net $(Q_\alpha)\subset \mathcal P_{f,0}(X,Y)$ with $L(Q_\alpha|_{B_X})\le 1$ satisfying $\widetilde{Q}_\alpha(x^{**})\to g(x^{**})$ for all $x^{**}\in B_{X^{**}}$. But for any $x\in B_X$ we have $\widetilde Q_\alpha (x)=Q_\alpha (x)\to g(x)=f(x)$. Now, $(iii)$ clearly  implies that $g=\widetilde f$.
\end{proof}

All the numbered results of Sections \ref{section: relation} and \ref{Section: local complementation} have easily adapted analogous versions for $\mathcal G$ and $\mathcal HL$ instead of $\mathcal G_0$ and $\mathcal HL_0$.

\subsection{The case of $\mathcal H^\infty(B_X)$ and $\mathcal G^\infty(B_X)$}
The arguments of this section can be canonically translated to prove analogous results for the case of $\mathcal G^{\infty}$ instead of $\mathcal G_0$ (and $\mathcal H^\infty$ instead of $\mathcal HL_0$). Moreover, for this case the hypothesis of symmetrical regularity is unnecessary. Let us state the results without proofs, since they are similar to the previous arguments.

\begin{theorem}
	If $X^{**}$ has the MAP then $\mathcal G^{\infty}(B_{X^{**}})$ is isometric to a 1-locally complemented subspace of $\mathcal G^{\infty}(B_X)^{**}$ and $\mathcal H^\infty(B_{X^{**}})$ is isometric to a 1-complemented subspace of $\mathcal H^\infty(B_X)^{**}$.
\end{theorem}

The following question is posed in \cite{CG}: when $X^{**}$ has the BAP, is it true that $\mathcal H^\infty(B_{X^{**}})$ is isomorphic to a complemented subspace of $\mathcal H^\infty(B_X)^{**}$?
Note that the previous theorem answers affirmatively this open question for the case $X^{**}$ having MAP.

\begin{theorem}
	Suppose  $X^{**}$ has the MAP. Consider a function
	$f\in\mathcal H^\infty(B_X)$ with
	$\|f\|=1$. Then, the following are
	equivalent:
	\begin{enumerate}[(i)]
		\item $f$ has a unique norm preserving extension to  $\mathcal H^\infty(B_{X^{**}})$.
		\item The AB extension from $(\overline B_{\mathcal H^\infty(B_X)}, w^*)$ to $(\overline B_{\mathcal H^\infty(B_{X^{**}})}, w^*)$ is continuous at
		$f$.
		\item If the net $(f_{\alpha})\subset  \overline B_{\mathcal H^\infty(B_X)}$ converges pointwise to $f$,
		then $(\widetilde f_\alpha)\subset  \overline B_{\mathcal H^\infty(B_{X^{**}})}$ converges pointwise to
		$\widetilde f$.
	\end{enumerate}
\end{theorem}

\section{Appendix}

Finally we will prove the following result as promised in Section \ref{sec:HL}.

\begin{theorem}\label{th:embedl_inftyother} There exists an isomorphism into  $F\colon \ell_\infty\to \mathcal H_0^\infty (\D)$ such that $F(\ell_\infty\setminus\{0\})\subset H_0^\infty (\D)\setminus \mathcal HL_0(\D)$ and $F(c_0\setminus\{0\})\subset \mathcal A(\D)\setminus \mathcal HL_0(\D)$. 
\end{theorem}

Note that one can easily prove a version for holomorphic functions on $B_X$ for any $X$ using the same ideas as in the proof of  Theorem \ref{th:embedl_infty}.

In what follows, we will use the function $\varphi_\lambda\colon \mathbb C\to\mathbb C$ given by
\[
\varphi_\lambda(z)=\frac{\overline{\lambda}z+1}{2}.
\]
It is a standard fact that
\begin{equation}\label{eq:varphi}\varphi_\lambda(\lambda)=1, \quad |\varphi_\lambda(z)|<1 \text{ for all } z\in \overline{\D}\setminus\{\lambda\}.
\end{equation} We also need the following technical lemma, which in particular provides another example of a non-Lipschitz function in the disc algebra $\mathcal A(\mathbb D)$.

\begin{lemma} \label{I} Fix $\lambda\in \C$ with $|\lambda|=1$ and define $f_\lambda\colon \C \to \C$ by
	\[f_\lambda(z)=\begin{cases}
		1+(\overline{\lambda}z-1)e^{1/(\overline{\lambda}z-1)} & \text{if $z\neq \lambda$}\\
		1 & \text{if $z= \lambda$}.
	\end{cases}
	\]
	Then
	\begin{enumerate}[(a)]
		\item\label{flambda_a} $f_\lambda$ is holomorphic in $\C\setminus\{\lambda\}$.
		\item\label{flambda_b} The restriction of $f_\lambda$ to $\overline{\D}$ belongs to $\mathcal{A}(\D)\setminus \mathcal HL(\D)$.
		\item\label{flambda_c} $|f_\lambda(z)|\leq 3$ for all $z\in \overline{\D}.$
		\item\label{flambda_d} If $0<s<1$, then $|f'_\lambda(z)|\leq \frac{s+1}{s}$ for all $z\in \D$ such that $|z-\lambda|\geq s$.
		\item \label{flambda_e} Given $k\in \N$ and $0<\delta<1$, we have that
		\[
		\sup_{z\in \D(\lambda,\delta)\cap\D}|\big(f_\lambda\cdot\varphi_\lambda^k\big)'(z)|=+\infty.
		\]
	\end{enumerate}
\end{lemma}
\begin{proof} 
	A standard computation shows that $\ref{flambda_a}$ holds. Now, to prove the rest of the claims it is enough to consider the case $\lambda=1$. Denote $f=f_1$ and take $z=a+ib\in \overline{\D}\setminus  \{1\}$, with $a,b\in \R$. We have that
	\[
	\Big|e^{\frac{1}{z-1}}\Big|=e^{\Re\frac{1}{z-1}}=e^{\frac{a-1}{(a-1)^2+b^2}}\leq e^0=1.
	\]
	Hence $f$, defined as $f(z)=1+(z-1)e^{\frac{1}{z-1}}$ is holomorphic on $\C\setminus \{1\}$ and continuously extends to $\overline{\D}$. Further $|f(z)|\leq 3$ for every  $z\in \overline{\D}$.
	Let us show that $f$ is not a Lipschitz function. For that, it is enough to check that $f'$ is not bounded  on $\D$. Taking a null sequence $0<\theta_n<1$ and setting $z_n:=\cos \theta_n(\cos \theta_n+i\sin \theta_n)$,  we obtain that the sequence $(z_n)\subset\D$ converges to $1$ and
	\[
	|f'(z_n)|=\Big|\frac{z_n-2}{z_n-1}\Big| e^{\Re\Big(\frac{1}{z_n-1}\Big)}=\Big|\frac{z_n-2}{z_n-1}\Big|e^{-1}.
	\]
	Consequently, $\lim_{n\to+\infty} |f'(z_n)|=+\infty$. Thus far we have proved $\ref{flambda_a}$, $\ref{flambda_b}$ and $\ref{flambda_c}$.  Let's check $\ref{flambda_d}$. We have
	\[
	|f'(z)|=\Big|\frac{z-2}{z-1}\Big|\cdot\Big|e^{\frac{1}{z-1}}\Big|\leq 1+\frac{1}{|z-1|} ,
	\]
	for all $z\in \D$. Hence, if $0<s<1$ and $z\in \D$ with $|z-1|\geq s$ we have that $|f'(z)|\leq \frac{s+1}{s}$.
	
	Finally  $\ref{flambda_e}$ is a consequence of   $(f\varphi^k)'(z)=f'(z)\varphi^k(z)+f(z)(\varphi^k)'(z)$ for all $z\in \C\setminus\{1\}$.
\end{proof}

\begin{proof}[Proof of Theorem \ref{th:embedl_inftyother}]
	To begin with, we choose a sequence $(\lambda_n)\subset \mathbb C \setminus\{1\}$ convergent to $1$ with $|\lambda_n|=1$ and $\lambda_n\neq \lambda_m$ for every $n\neq m$.
	Consider the function $\Phi\colon\C^2\to \C$ and $\varphi_n\colon\C\to \C$, $\varphi_n(z):=\Phi(z, \lambda_n)$ defined as
	$$
	\Phi(z,\lambda)=\frac{\overline{\lambda}z+1}{2}
	$$
	and, for each $p\in\N$, the compact subset of $\C^2$
	\[
	K_p = \{(\lambda_p,\lambda_n): n\in \mathbb N, n\neq p\}\cup\{(\lambda_p,1)\}.
	\]
	We have $
	\big|\Phi(z,\lambda)\big|<1$
	for every $(z,\lambda)\in K_p$ by \eqref{eq:varphi}, and $\Phi$ is continuous on $\C^2$. Hence, there exists $0<s_p<1$ such that $
	\big|\Phi(z,\lambda)\big|<1$   for every $(z,\lambda)\in K_p+\overline{\D}((0,0),s_p)$. In particular,
	\begin{equation}\label{II}
		|\varphi_n(z)|=\big|\Phi(z,\lambda_n)\big|<1,
	\end{equation}
	for all $z\in \overline{\D}(\lambda_p,s_p)$ and all $n\neq p$.
	
	Now, since the sequence $(\lambda_n)$ is convergent to $1$ we can find a sequence of positive numbers $(r_n)$ that tends to $0$  such that $0<2r_n<s_n$ for all $n\in \mathbb N$ and such that $  \overline{\D}(\lambda_n, 2r_n)\cap \overline{\D}(\lambda_p, 2r_p)=\emptyset,$
	for all $n\neq p$. Moreover, as $(r_n)$ converges to 0,  for each $n\in \mathbb N$ the set
	\[
	L_n:=\bigcup_{p\neq n} \overline{\D}(\lambda_p, 2r_p)\cup\{1\},
	\]
	is also a compact subset of $\C$, (although it is not a subset of $\overline{\D}$) and $|\varphi_n(z)|<1$  for all $z\in L_n$. Since $|\varphi_n|$ is continuous on $\C$ we obtain that
	\[
	\max \{|\varphi_n(z)|\,:\, z\in C_n\cup L_n\}<1,
	\]
	for all $n$, where $C_n=\overline{\D}\setminus \D(\lambda_n, r_n)$.
	As a consequence, for each $n$ the sequence $\Big(\varphi_n^k\Big)_{k=1}^\infty$ converges uniformly to 0 on $C_n\cup L_n$ and we can find a $k_n\in \N$ such that
	\begin{equation}\label{III}
		|\varphi_n^{k_n}(z)|<\frac{r_n}{3^{n+1}},
	\end{equation}
	for every $z\in C_n\cup L_n$.
	
	We denote $f_n:=f_{\lambda_n}$, for $n\in \mathbb N$ and we define $F\colon \ell_\infty\longrightarrow \mathcal{H}^\infty(\D)$ by
	\[
	F(a_n):=\sum_{n=1}^{\infty}a_nf_n\varphi_n^{k_n},
	\]
	For each $(a_n)\in \ell_\infty$ the series $F(a_n)(z)$ is  convergent for each $z\in \overline{\D}$. To see this, we first suppose that
	
	a) $z\in \overline{\D}\setminus\left(\bigcup_{n=1}^\infty \overline{\mathbb D}(\lambda_n,r_n)\right)$. In that case, by \eqref{III} and Lemma \ref{I}. $\ref{flambda_c}$,
	\begin{equation}\label{IV}
		\sum_{n=1}^{\infty}|a_nf_n(z)\varphi_n^{k_n}(z)|\leq \sum_{n=1}^{\infty} 3|a_n|\frac{r_n}{3^{n+1}}\leq \frac{1}{2}\|(a_n)\|_\infty.
	\end{equation}
	Hence $F(a_n)(z)$ converges. Moreover,  the series $F(a_n)$ converges absolutely and uniformly on the open set  $\D\setminus \Big(\bigcup_{n=1}^\infty \overline{\D}(\lambda_n,r_n)\Big)$. Thus $F(a_n)$ is holomorphic in that open set.

	If this does not occur, then it must be that we have:
	
	b)  There exists a unique $n_0\in \N$ such that $z\in \D(\lambda_{n_0}, 2r_{n_0})$. By \eqref{III}, for every $u\in \D(\lambda_{n_0}, 2r_{n_0})$ we have that
	\[
	|a_nf_n(u)\varphi_n^{k_n}(u)|\leq  3|a_n|\frac{r_n}{3^{n+1}}<\frac{|a_n|}{3^n},
	\]
	for all $n\neq n_0$ and
	\[
	|a_{n_0}f_{n_0}(u)\varphi_{n_0}^{k_{n_0}}(u)|\leq 3|a_{n_0}|.
	\]
	Hence,
	\begin{equation}\label{V}
		\sum_{n=1}^{\infty}|a_nf_n(z)\varphi_n^{k_n}(z)|\leq 4\|(a_n)\|_\infty,
	\end{equation}
	and we have obtained that for every $z\in\D(\lambda_{n_0}, 2r_{n_0})$, $F(a_n)(z)$ exists and in fact $|F(a_n)(z)|\leq 4\|(a_n)\|_\infty$. But our argument shows that the series  $F(a_n)$ is  absolutely and uniformly convergent in the open disc $\D(\lambda_{n_0}, 2r_{n_0})$. Hence, $F(a_n)$ is holomorphic  on $\D\cup \bigcup_{n=1}^\infty \D(\lambda_n,2r_n)$ and
	$F\colon \ell_\infty\to\mathcal{H}^\infty(\D)$ is a  continuous linear mapping since $  \|F(a_n)\|\leq 4\|(a_n)\|_\infty$
	for all $(a_n)\in \ell_\infty$.
	
	Now we check that $F$  is bounded below. We already know that for each $(a_n)\in \ell_\infty$, the function  $F(a_n)$ is holomorphic on $\D\cup \bigcup_{n=1}^\infty \D(\lambda_n,2r_n)$ and bounded on $\D$. Thus, using \eqref{III} and the fact that $\lambda_p\in \overline{\mathbb D}$, we get
	\begin{align*}
		\|F(a_n)\|   &= \sup_{z\in \D}|F(a_n)(z)|  \geq \sup_{p\in \N}|F(a_n)(\lambda_p)|  \geq \sup_{p\in \N} \left\{|a_p|-\sum_{n\neq p} 3|a_n|\frac{r_n}{3^{n+1}}\right\}\\
		& \geq \sup_{p\in \N} \left\{|a_p|-\frac{\|(a_n)\|_{\infty}}{2}\right\}=\frac{\|(a_n)\|_{\infty}}{2}
	\end{align*}
	for every $(a_n)\in \ell_\infty$.
	\medskip

	Let's  check that if $(b_n)\in c_0$, then $F(b_n)$ belongs to $\mathcal{A}(\D)$.     Given $\varepsilon >0$, there exists $n_1\in \N$ such that $  |b_n|<\frac{\varepsilon}{3}$,  for every $n\geq n_1$. Thus, if $z\in
	\overline{\D}$.
	\begin{equation}\label{VI}
		\sum_{n=n_1}^{\infty}|b_nf_n(z)\varphi_n^{k_n}(z)|\leq 3\varepsilon \sum_{n=n_1}^{\infty} |\varphi_n^{k_n}(z)|.
	\end{equation}
	Now if,   $z\in \overline{\D}\setminus \Big(\bigcup_{n=1}^\infty \overline{\D}(\lambda_n,r_n)\Big)$,  then by\eqref{III}, $ |\varphi_n^{k_n}(z)| \leq \frac{r_n}{3^{n+1}}$. Hence, by \eqref{VI},
	\[
	\sum_{n=n_1}^{\infty}|b_nf_n(u)\varphi_n^{k_n}(z)|< \varepsilon.
	\]
	Otherwise, if $z\in\overline{\D}\cap \Big(\bigcup_{n=1}^\infty \overline{\D}(\lambda_n,r_n)\Big)$, there is a unique $n_0\in \N$ such that
	$z\in \overline{\D}(\lambda_{n_0}, r_{n_0})$ and
	\begin{equation*}
		\sum_{n=n_1}^{\infty}|b_nf_n(z)\varphi_n^{k_n}(z)|\leq \varepsilon+ \sum_{\stackrel{n=n_1}{n\neq n_0}}^{\infty} \varepsilon\frac{r_n}{3^{n+1}}<2\varepsilon.
	\end{equation*}
	Consequently the series $\sum_{n=1}^{\infty}b_nf_n(z)\varphi_n^{k_n}(z)$ converges absolutely and uniformly on $\overline{\D}$ and  $F_{|c_0}\colon c_0\to \mathcal{A}(\D)$ is a well-defined continuous linear mapping.
	
	Consider $(a_n)\in \ell_\infty\setminus \{0\}$. There exists $n_0$  such that $a_{n_0}\neq 0$. We are going to show that $F(a_n)'(z)$ is not bounded on $\D(\lambda_{n_0}, \frac{r_{n_0}}{3})\cap \D$.
	
	By the Weierstrass theorem,
	\[
	F(a_n)'(z)=\sum_{n=1}^{+\infty}a_n\big(f_n\varphi_n^{k_n}\big)'(z),
	\]
	for every $z\in \D\cup \bigcup_{n=1}^\infty \D(\lambda_n,2r_n)$.
	If $n\neq n_0$, then by the Cauchy integral formula
	\[
	\big(\varphi_n^{k_n}\big)'(z)=\frac{1}{2\pi i}\int_{C(\lambda_{n_0}, r_{n_0})}\frac{\varphi_n^{k_n}(u)}{(u-z)^2}du,
	\]
	for every $z\in \D(\lambda_{n_0}, \frac{r_{n_0}}{3})$.
	Thus, by \eqref{II} and \eqref{III}, we obtain
	\[
	\begin{split}
		\sup_{z\in \D(\lambda_{n_0}, \frac{r_{n_0}}{3})}|\big(\varphi_n^{k_n}\big)'(z)| &  \leq \frac{r_{n_0}}{(\frac{2}{3}r_{n_0})^2}\sup_{|u-\lambda_{n_0}|= r_{n_0}}|\varphi_n^{k_n}(u)|  <\frac{9}{4r_{n_0}}\frac{r_n}{3^{n+1}}<\frac{1}{r_{n_0}}\frac{1}{3^{n}},
	\end{split}
	\]
	and we get
	\[    |\big(f_n\varphi_n^{k_n}\big)'(z)|  \leq |f'_n(z)||\varphi_n^{k_n}(z)|+|f_n(z)||(\varphi_n^{k_n})'(z) <\frac{1}{3^n}+\frac{1}{r_{n_0}}\frac{1}{3^{n-1}},
	\]
	where in the second inequality we have applied, \eqref{II}, \eqref{III} and the properties of $f_n$ and $f'_n$ given in Lemma \ref{I}. Hence,
	\[
	|F(a_n)'(z)|\  \geq  |a_{n_0}||\big(f_{n_0}\varphi_{n_0}^{k_{n_0}}\big)'(z)|-\|(a_n)\|_{\infty}\big( \frac{1}{2}+\frac{3}{2r_{n_0}}\big),\\
	\]
	for every $z\in \D(\lambda_{n_0}, \frac{r_{n_0}}{3})$. Finally, by Lemma \ref{I}.$\ref{flambda_e}$, we have that
	$F(a_n)'$ is unbounded on $\D(\lambda_{n_0},\frac{r_{n_0}}{3})\cap\D$ and hence, $F(a_n)$ does not belong to $\mathcal{H}L(\D)$.
	
	Finally, if we define $F_1\colon\ell_\infty \to \mathcal H^\infty_0(\D)$ by
	$F_1(a_n)(z):=zF(a_n)(z)$ for $(a_n)\in \ell_\infty$ and $z\in \D$, it is clear that $F_1$ is an isomorphism onto its image and that   $F_1(\ell_\infty\setminus \{0\})\subset  \mathcal{H}^\infty_0(\D)\setminus \mathcal{H}L(\D)$.
	
\end{proof}

Finally, we note that if we are only interested in $\mathcal{A}(\D)$ there are known results  related to Theorem~\ref{th:embedl_inftyother}. Indeed, in three relevant papers \cite{BLS, BBLS2018,BBLS2019}, L. Bernal et al.  have obtained  many  results on the existence of large subspaces of functions that belong to $\mathcal{A}(\D)\setminus \mathcal HL(\D)\cup \{0\}$. In particular, in \cite[Th. 4.1.c]{BLS} the authors show that there exists an infinite dimensional Banach space $X$ contained in $\mathcal{A}(\D)$ such that any non-null function in $X$ is not differentiable on any point of a  fixed dense subset of $\mathbb{T}$. Also, in \cite[Th. 3.4]{BBLS2019}, they prove that there exists  an infinite dimensional Banach space $X$, contained in  $\mathcal{A}(\D)$, (which, however, is endowed with a stronger norm than the one inherited from  $\mathcal{A}(\D)$) such that if $f \in X$, then  the restriction of $f$ to $\mathbb T$ is nowhere H\"older on $\mathbb T$.

\subsection*{Acknowledgments} This paper had its start during the summer, 2022 in Banff, thanks to the BIRS Research in Teams program. The research was also partially supported by Agencia Estatal de Investigaci\'{o}n and EDRF/FEDER ``A way of making Europe" (MCIN/AEI/10.13039/501100011033) through grants PID2021-122126NB-C32 (Garc\'ia-Lirola) and PID2021-122126NB-C33 (Aron and Maestre). The research of Garc\'ia-Lirola was also supported by DGA project E48-20R, Fundaci\'on S\'eneca - ACyT Regi\'on de Murcia project 21955/PI/22 and Generalitat Valenciana project CIGE/2022/97. The research of Dimant was partially supported by CONICET PIP 11220200101609CO and ANPCyT PICT 2018-04104. The research of Maestre also was partially supported by PROMETEU/2021/070. We are very grateful to A.~Procházka for his comments that allowed us to shorten the proof of 
Theorem~\ref{teo: equivalencia AP} significantly, and to A.~Chávez-Domínguez for helpful conversations about Section \ref{sec:AP}.

\end{document}